\documentclass[11pt]{article}
\usepackage[utf8]{inputenc}
\usepackage[english]{babel} 
\usepackage[T1]{fontenc} 
\usepackage[utf8]{inputenc} 
\usepackage{soul}
\usepackage[normalem]{ulem}
\usepackage{fancybox}
\usepackage{fancyvrb}
\usepackage{listings}
\usepackage{moreverb}
\usepackage{url}
\usepackage{graphicx}
\graphicspath{ {./images/} }
\usepackage{textcomp}
\usepackage{lmodern}
\usepackage{eurosym}
\usepackage[cyr]{aeguill}
\usepackage{amsmath}
\usepackage{wrapfig} 
\usepackage{makeidx} 
\usepackage{picinpar} 
\usepackage[final]{pdfpages} 
\usepackage{array,multirow,tabularx}
\usepackage{amssymb}
\usepackage[a4paper]{geometry}
\usepackage{amsthm}
\usepackage{tikz}
\usepackage{fancyhdr}
\usepackage{colortbl}
\usepackage{color}
\usepackage{setspace}
\usepackage{longtable}
\usepackage{hhline}
\usepackage{arydshln}
\usepackage{makecell}
\usepackage{mathtools}
\usepackage{tkz-tab}
\usepackage{tikz-cd}
\usepackage{multicol}
\usepackage{enumerate} 
\tikzset{node distance=2.0cm, auto}

\newcommand*\quot[2]{{^{\textstyle #1} \big/_{\textstyle #2}}}
\newcommand{\black}{\color{black}}

\usepackage[section]{placeins}

\DeclareMathOperator{\Hom}{Hom}

\DeclareMathOperator{\Ext}{Ext}

\DeclareMathOperator{\h}{\mathcal{H}}
\DeclareMathOperator{\K}{\mathbb{K}}

\DeclareMathOperator{\N}{\mathbb{N}}
\DeclareMathOperator{\F}{\mathbb{F}}
\DeclareMathOperator{\ce}{\mathrm{CE}}

\DeclareMathOperator{\Der}{Der}

\DeclareMathOperator{\ad}{ad}
\DeclareMathOperator{\g}{\mathfrak{g}}

\DeclareMathOperator{\im}{\text{Im}}
\DeclareMathOperator{\Ker}{\text{Ker}}
\DeclareMathOperator{\ind}{\text{ind}}
\DeclareMathOperator{\w}{\omega}
\DeclareMathOperator{\id}{\text{id}}
\DeclareMathOperator{\obs}{\mathrm{obs}}
\DeclareMathOperator{\Obs}{\mathfrak{obs}}
\DeclareMathOperator{\Sl}{\mathfrak{sl}}
\DeclareMathOperator{\fd}{\mathfrak{d}}
\DeclareMathOperator{\fC}{\mathfrak{C}}
\DeclareMathOperator{\fZ}{\mathfrak{Z}}
\DeclareMathOperator{\fB}{\mathfrak{B}}
\DeclareMathOperator{\fH}{\mathfrak{H}}
\allowdisplaybreaks
	\newtheorem{thm}{Theorem}[section]
	\newtheorem{prop}[thm]{Proposition}
	\newtheorem{lem}[thm]{Lemma}
	
	\theoremstyle{definition}
	\newtheorem{defi}[thm]{Definition}
	\newtheorem*{rmq}{Remark}

\title{Cohomology and deformations of restricted Lie algebras and their morphisms in  positive characteristic}
\author{Quentin Ehret\footnote{Université de Haute-Alsace, IRIMAS UR 7499, F-68100 Mulhouse, France; and Division of Science and Mathematics, New York University Abu Dhabi, P.O. Box 129188, Abu Dhabi, United Arab Emirates.\\ E-mail: \texttt{qe209@nyu.edu}. QE was partially supported by the grant NYUAD-065.}  \; and  Abdenacer Makhlouf\footnote{Université de Haute-Alsace, IRIMAS UR 7499, F-68100 Mulhouse, France.\\ E-mail:  \texttt{abdenacer.makhlouf@uha.fr}.} }
\date{\today}

\begin{document}

	\maketitle
	\begin{abstract}
		The main purpose of this paper is  to study cohomology and develop a deformation theory of  restricted Lie algebras in positive characteristic $p>0$.  In the case $p\geq3$,  it is shown that the deformations of restricted Lie algebras are controlled by the restricted cohomology introduced by Evans and Fuchs. Moreover, we introduce a new cohomology that controls the deformations of restricted morphisms of restricted Lie algebras. In the case $p=2$, we provide a full restricted cohomology complex with values in a restricted module and investigate its connections with formal deformations. Furthermore, we introduce a full deformation cohomology that controls deformations of restricted morphisms of  restricted Lie algebras in characteristic $2$. As example, we discuss restricted cohomology with adjoint coefficients of restricted Heisenberg Lie algebras in characteristic $p\geq 2$.
	\end{abstract}
	
	\noindent\textbf{Keywords:} Modular Lie algebra, restricted cohomology, deformation.\\
	\noindent\textbf{MSC 2020 classification:} 17B50, 17B56, 17B60.
	\tableofcontents
	\section{Introduction}

This paper includes the restricted Lie algebras results part of the Preprint \cite{EM23}, which is dedicated to Lie-Rinehart algebras in positive characteristic. We consider here the study of cohomology and deformations of restricted Lie algebras in positive characteristic. Moreover, we consider deformations of restricted morphisms of restricted Lie algebras.\\

\noindent\textbf{Restricted Lie algebras.} Lie algebras were historically introduced over the field of complex numbers, then over arbitrary fields of characteristic zero. The origins of the study of Lie algebras in positive characteristic $p>0$ go back to the late 1930s, with the discovery by Witt in 1937 of a new simple Lie algebra, named after him and generalized by Zassenhaus in 1939 (\cite{Z39}). Many results and tools that are valid in characteristic zero are no longer valid in positive characteristic, such as the Killing form, Lie Theorem and Weyl Theorem. This makes the classification problem in characteristic $p$ a difficult one. The classification of simple Lie algebras in characteristic $p$ has been investigated, for example by Strade in a series of six articles published between 1989 and 1998, see \cite{SH98} and  references therein. More recently, Bouarroudj and his collaborators considered superalgebras cases, see (\cite{BGL09,BKLLS,BLLS}).
In characteristic $p>0$, an additional structure appears naturally on certain Lie algebras, inspired by the following fact.
If $A$ is an associative algebra over a field $\F$ of positive characteristic $p$, we can consider its Lie algebra of derivations $\Der(A)$.  Then the Frobenius morphism $(\cdot)^p:\Der(A)\rightarrow\text{End}(A),~D\mapsto D^p$ is actually an endomorphism of $\Der(A)$, which is in general not true in characteristic zero. This observation, together with the study of the properties arising from the interactions between this Frobenius morphism and the structuring applications of $\Der(A)$, led to the definition of \emph{restricted Lie algebra} (Jacobson, \cite{JN37,JN41}, see Definition \ref{restdefi}), which is a Lie algebra $L$ equipped with a so-called $p$-map $(\cdot)^{[p]}: L\rightarrow L$ satisfying some compatibility conditions with the Lie bracket and the additive law of $L$. For example, Lie algebras associated with algebraic groups over fields of positive characteristic are also restricted. Restricted Lie algebras are  interesting to study in characteristic $p$ because they allow us to develop new techniques that partially overcome the problems mentioned above.\\

\noindent\textbf{Restricted cohomology.} The cohomology associated with restricted Lie algebras is much more complicated than the ordinary Chevalley-Eilenberg cohomology in characteristic $0$.
In \cite{HG54,HG55.1,HG55.2}, Hochschild defines a restricted cohomology of a restricted Lie algebra $L$ with values in a module $M$ by
$$
    H^k_{res} (L,M) := \Ext^k_{U_p(L)}(F,M),~k\in \N,
$$
where $U_p(L)$ denotes the \textit{restricted} enveloping algebra of $L$. Although correct, this expression only allows explicit calculations for $k\in\{0,1\}$, in the context of certain extensions (see \cite{HG54}). For $p\geq3$, Evans and Fuchs then proposed an explicit construction of a cochain complex in \cite{ET00, EF08}, which allows restricted cohomology groups to be computed up to order $p$ if the Lie algebra is abelian and up to order $2$ in the general case. The complete cohomology is still a challenging problem, although work of Evans and Fuchs has provided good cohomological interpretations of certain algebraic phenomena (see \cite{BE24, E23, EF08, EFP,EF2,EF3, EFY24}).\\

\noindent\textbf{Formal deformations.} Formal deformations were introduced by Gerstenhaber in \cite{GM64} for associative algebras, then generalized for various algebraic structures, most notably for Lie algebras by Nijenhuis and Richardson (\cite{NR66,NR67}). The main tool is to consider formal power series. Roughly speaking, a deformation of an algebraic structure $(A,\mu)$ consists in building a multiplicative operation $\mu_t$ on the formal space $A[[t]]$ of the form $\mu_t=\mu+\sum_{i\geq1}t^i\mu_i,$ where the maps $\mu_i:A\times A\rightarrow A$ are bilinear and must satisfy a system of conditions called \textit{deformation equation}. Formal deformations of $A$ are controlled by the second cohomology space with coefficients in $A$. Formal deformations of morphisms were considered by Gerstenhaber and Schack for associative algebras in \cite{GS83,GS85}, by Nijenhuis and Richardson for Lie algebras in \cite{NR67, NR67.2} and by Mandal for Leibniz algebras in \cite{Ma07}. Most of the studies dealt with characteristic $0$. In \cite{EF08}, Evans and Fuchs sketched a deformation theory for restricted Lie algebras in characteristic $p>0$ as an application of their cohomology formulas. In this paper, we aim to develop a deformations theory of restricted Lie algebras and their morphisms. A connection with the restricted cohomology is also explored.  \\

\noindent\textbf{Characteristic 2.} In the specific case where the characteristic of the ground field is equal to $2$, many results fall short and new techniques are required. For example, the Lie algebra $\mathfrak{sl}_2(\F)$ is the standard example of a simple Lie algebra in characteristic $p\neq 2$. But this is no longer the case in characteristic $2$: this algebra admits a non-zero center and is nilpotent. Bouarroudj and his collaborators have made major contributions to the study of the special case $p=2$, about double extensions in \cite{BB18}, deformations in \cite{BLLS15} and classification of simple Lie superalgebras in \cite{BGL09, BLLS}. Recently, Bouarroudj and Makhlouf have studied (Hom-)Lie superalgebras in characteristic $2$ (\cite{BM22}), where a new cohomology is introduced. It appears that in the case $p=2$, there are similarities between the notions of restricted Lie algebra and Lie superalgebra. This observation motivated the construction of a new cochain complex for restricted Lie algebras, which has no analogue for $p\neq2$ (see Section \ref{seccoh2}). This complex is complete in the sense that it allows the computation of restricted cohomology groups of any order. However, this method remains specific to the case $p=2$ and cannot be generalized to $p>2$.\\

\noindent\textbf{Outline of the paper.} In this paper, we first recall basic notions about restricted Lie algebras in Section \ref{Restricted Lie Algebras}, as well as restricted cohomology formulas introduced by Evans and Fuchs for $p\geq3$. Section \ref{Deformation theory of restricted Lie algebras} is devoted to the deformation theory of restricted Lie algebras in characteristic $p\geq3$. We show that infinitesimals of a restricted deformation are restricted 2-cocycle (Theorem \ref{restdefo}). We investigate equivalence classes of restricted deformations (Theorem \ref{restequiva} and Proposition \ref{trivrest}) and study obstructions to the extension of deformations of order $n$ to order $n+1$, see Section \ref{sectionobstructions}. We also introduce a (partial) deformation cohomology for restricted morphisms and show that it fits with their restricted deformations, see Section \ref{pmorphdefo}. In Section \ref{R2}, we investigate the particular case of characteristic $2$. We provide a restricted cochain complex and restricted differentials (Theorem \ref{cohomology2}), which allow to compute restricted cohomological groups at any order. We then provide algebraic interpretations of this cohomology (Section \ref{compu2}) and study restricted deformations in Section \ref{defo2section}.  Moreover, we introduce a (full) deformation cohomology for restricted morphisms in Section \ref{2morpfdefo2}. The restricted cohomology in characteristic $p=2$ allows us to have more general results than in characteristic $p\geq3$. Finally, we compute explicitly the second restricted cohomology groups with adjoint coefficients for the restricted Heisenberg Lie algebras in Section \ref{sectionh}. For that purpose, we first classify all restricted structures on the Heisenberg Lie algebras of dimension $3$ (Theorems \ref{heisclass} and \ref{heisclass2}), then compute basis for the restricted cohomology spaces (Theorems \ref{cohogroupp}, \ref{cohogroups3} and \ref{cohogroup2}). 

~\\	
Throughout the paper, ``ordinary" shall be understood as ``not restricted".
	
	
	\section{Restricted Lie Algebras}\label{Restricted Lie Algebras}
	
	We first review some basics about  restricted Lie algebras and  their cohomology.
	
	\subsection{Basics}\label{basics}
	
	Let $\F$ denote a field of characteristic $p\neq 0$. For a comprehensive introduction to the  notions introduced here, we refer to \cite{SF88}, see also \cite{JN41}.
	
	\begin{defi}[Restricted Lie Algebra]\label{restdefi}
		A \textit{restricted Lie algebra} over $\F$ is a Lie $\F$-algebra $\bigl(L,[\cdot,\cdot]\bigl)$ endowed with a map $(\cdot)^{[p]}:L\longrightarrow L$ such that
		\begin{enumerate}
			\item[$(i)$] $(\lambda x)^{[p]}=\lambda^px^{[p]}$, $\forall x\in L$, $\forall \lambda\in \F$;
			\item[$(ii)$] $\bigl[x,y^{[p]} \bigl]=[ [\cdots[x,  \overset{p\text{ terms}}{\overbrace{y],y],\cdots,y}}],~\forall x,y\in L$;
			\item[$(iii)$] $(x+y)^{[p]}=x^{[p]}+y^{[p]}+\displaystyle\sum_{i=1}^{p-1}s_i(x,y),~\forall x,y\in L$,
		\end{enumerate}
		where $is_i(x,y)$ is the coefficient of $Z^{i-1}$ in $\ad^{p-1}_{Z x+y}(x)$. Such a map $(\cdot)^{[p]}:L\longrightarrow L$ is called a $p$-map.
	\end{defi}
	\noindent We have an explicit formula
	\[
	is_i(x,y)=\sum_{\underset{\sharp\{k,~x_k=x\}=i-1}{x_k\in\{x,y\}}}[x_1,[x_2,[\cdots,[x_k,\cdots,[x_{p-2},[y,x]]\cdots], \]
	where $\sharp\{k,~x_k=x\}$ refers to the number of $x_k$'s equal to $x$. We refer to a restricted Lie algebra by a triple $(L,[ \cdot , \cdot ],(\cdot )^{[p]})$.
	\bigskip
	
	\noindent Throughout the paper, we denote by $\sharp\{x\}$ the number of $x$'s among the $x_k$'s. Then we have
	\begin{align}\label{si}
		\sum_{i=1}^{p-1}s_i(x,y)&=\sum_{i=1}^{p-1}\frac{1}{i}\sum_{\underset{\sharp\{k,~x_k=x\}=i-1}{x_k\in\{x,y\}}}[x_1,[x_2,[\cdots,[x_k,\cdots,[x_{p-2},[y,x]]\cdots]\nonumber\\&=\sum_{\underset{x_{p-1}=y,~x_p=x}{x_k\in\{x,y\}}}\frac{1}{\sharp\{x\}}[x_1,[x_2,[\cdots,[x_k,\cdots,[x_{p-1},x_p]]\cdots],
	\end{align}
	since $\frac{1}{i}$ is exactly the inverse of the number of $x$'s among the $x_k$'s.
	
	\bigskip
	
	We have the following particular cases:
	
	\begin{itemize}
		\item[$\bullet$] \underline{$p=2$}:  for all $x,y\in L$, we have $\ad_{Z x+y}(x)=[x,Z x+y]=[x,y]$. Then $s_1(x,y)=[x,y]$. Hence,
  $$(x+y)^{[2]}=x^{[2]}+y^{[2]}+[x,y].$$
		
		\item[$\bullet$] \underline{$p=3$}: for all $x,y\in L$, we have $\ad^2_{Z x+y}(x)=\left[[x,Z x+y],Z x+y \right]=Z[[x,y],x]+[[x,y],y]$. It follows that  $s_1(x,y)=[[x,y],y]$ and $ s_2(x,y)=2[[x,y],x]$. Hence, 
		$$(x+y)^{[3]}=x^{[3]}+y^{[3]}+[[x,y],y]+2[[x,y],x].$$
		
	\end{itemize}
	
	\begin{rmq}
		Recall that the center of a Lie algebra $L$ is defined by $Z(L)=\{x\in L,~\ad_x=0\}$. If the adjoint representation $\ad:x\mapsto [x,\cdot]$ is faithful (or equivalently, if the Lie algebra is centerless), then both Conditions $(i)$ and $(iii)$ in Definition \ref{restdefi} follow from  Condition $(ii)$.
	\end{rmq}

	\noindent\textbf{Examples:} 
	\begin{enumerate}
		\item Let $A$ be an associative algebra over $\F$. Endowed  with the bracket $[x,y]=xy-yx$,  the vector space  $A$ becomes a restricted Lie algebra with the map $x\longmapsto x^p$,
		called \textit{Frobenius morphism}.
		
		\item Let $L$ be an abelian Lie algebra. Then, any map $f: L\rightarrow L$ satisfying $$f(\lambda x+y)=\lambda^p f(x)+f(y),~\forall x,y\in L,~\forall \lambda\in \F$$ is a $p$-map on $L$. A map satisfying such a property is called \textit{$p$-semilinear}.
		
		
		\item Let $\mathbb{F}$ be a field of characteristic $p\geq5$. We consider the \textit{Witt algebra} $W(1)=\text{Span}_{\F} \{e_{-1},e_0,...,e_{p-2}\}$   with the bracket
		$$[e_i,e_j]=
		\begin{cases}
			(j-i)e_{i+j} ~\text{ if }~ i+j\in\{-1,...,p-2\};\\
			0 ~\text{ otherwise;}\\
		\end{cases}$$
		and the $p$-map
		$$e_i^{[p]}=
		\begin{cases}
			e_0~\text{ if } i= 0.;\\
			0 ~\text{ if } i\neq 0.\\
		\end{cases}$$
		
		Then $\left(W(1),[\cdot,\cdot], (\cdot)^{[p]}\right)$ is a restricted Lie algebra (see \cite{EFP}). Moreover, this Lie algebra is also simple, so the restricted structure is unique. 
		
		\noindent One can see the Witt algebra as the derivations algebra of the commutative associative algebra 
		$A:=\quot{\F[x]}{(x^p-1)}$ (see \cite{EF02}). In this setting, the basis elements are $e_i=x^{i+1}\frac{d}{dx}$, the bracket being the commutator: if $f\in A,$ we have $$\left[x^{i+1}\frac{d}{dx},x^{j+1}\frac{d}{dx}\right](f)=(j-i)x^{i+j+1}\frac{df}{dx}\text{ if } i+j+1\in\{-1,...,p-2\} \text{ and } 0 \text{ otherwise. }$$  The $p$-map is then given by $$\left(x\frac{d}{dx}\right)^{[p]}=x\frac{d}{dx} \text{ and } \left(x^k\frac{d}{dx}\right)^{[p]}=0,~k\neq 1.$$
		

	\end{enumerate}
	
	\begin{defi}
		Let $\left( L_1,[\cdot,\cdot]_1,(\cdot)^{[p]_1}\right) $ and $\left( L_2,[\cdot,\cdot]_2,(\cdot)^{[p]_2}\right) $ be two restricted Lie algebras. A \textit{restricted morphism} (or \textit{$p$-morphism}) $\varphi:L_1\longrightarrow L_2$ is a Lie algebra morphism that satisfies $\varphi\bigl(x^{[p]_1} \bigl)=\varphi(x)^{[p]_2}$, $\forall x\in L_1$. 
	\end{defi}
	
	\begin{defi}
		
		Let $M$ be an $L$-module over a restricted Lie algebra $\left( L,[\cdot,\cdot],(\cdot)^{[p]}\right)$, that is, $M$ is endowed with an action $L\times M\rightarrow M$ such that $[x,y]\cdot m=x\cdot (y\cdot m)-y\cdot (x\cdot m)$, for all  $x,y\in L$ and all $m\in M$. The $L$-module is called \textit{restricted} if we have, in addition, $x^{[p]}\cdot m=\overset{p\text{ terms}}{\overbrace{x\cdot(x\cdots (x}}\cdot m)\cdots)$, for all $m\in M$ and all $ x\in L$. 
		
	\end{defi}
	
	\begin{thm}[Jacobson's Theorem (\cite{JN62})]\label{jacobson}
		Let $L$ be a $n$-dimensional Lie algebra over a field $\F$ of characteristic $p$. Suppose that $(e_j)_{j\in \{1,\cdots, n\}}$ is a basis of $L$ such that it exists $y_j\in L,~\bigl(\ad_{e_j} \bigl)^p=\ad_{y_j}$. Then it exists exactly one $p$-map such that $e_j^{[p]}=y_j,~\forall j=1,\cdots, n$. 
	\end{thm}

	
	\subsection{Cohomology of Restricted Lie Algebras, $p\geq 3$}
	
	We assume here that  the ground field $\F$ is of characteristic $p>2$. We recall  the  Chevalley-Eilenberg cohomology complex for ordinary Lie algebras (\cite{CE48}) and the restricted cohomology for restricted Lie algebras defined in \cite{EF08}, where   restricted cochains are considered up to order $3$ and restricted coboundary maps up to order $2$.
	
	\subsubsection{Ordinary Chevalley-Eilenberg Cohomology}
	
	Let $L$ be a Lie algebra and $M$ be an $L$-module. Let  $m\geq 1$, and $\displaystyle TL=\bigoplus T^mL$ be the tensor algebra of $L$. We set
	\[ \Lambda^mL=T^mL/\left\langle x_1\otimes\cdots\otimes x_{k}\otimes x_{k+1}\otimes\cdots \otimes x_m+x_1\otimes\cdots\otimes x_{k+1}\otimes x_k\otimes\cdots \otimes x_m \right\rangle,~~~x_1,\cdots,x_m \in  L.    \]
	We define the cochains 
	\begin{align*}
		C^m_{\text{CE}}(L,M)&=\Hom_{\F}(\Lambda^mL, M)~\text{ for all }m \geq 1,\\
		C^0_{\text{CE}}(L,M)&\cong M.
	\end{align*}
	We define a differential map $d^m_{\text{CE}}: C^m_{\text{CE}}(L,M)\longrightarrow C^{m+1}_{\text{CE}}(L,M)$ by
	\begin{align*}
		d_{\text{CE}}^m(\varphi)(x_1,\cdots,x_{m+1})&=\sum_{1\leq i<j\leq m+1}(-1)^{i+j}\varphi\left([x_i,x_j],x_1,\cdots\hat{x_i},\cdots,\hat{x_j},\cdots, x_{m+1} \right)\\&+\sum_{i=1}^{m+1}(-1)^{i+1}x_i\varphi(x_1,\cdots,\hat{x_i},\cdots,x_{m+1}),   
	\end{align*}
	where $\hat{x_i}$ means that the element is omitted. 
	We have $d^{m+1}_{\text{CE}}\circ d_{\text{CE}}^{m}=0$. We denote  the $m$-cocycles by $Z^m_{\text{CE}}(L,M)=\Ker(d_{\text{CE}}^m)$ and the $m$-coboundaries by $B^m_{\text{CE}}(L,M)=\im(d_{\text{CE}}^{m-1})$.
	Then we define the ordinary Chevalley-Eilenberg cohomology of $L$ with values in $M$ by
	$$  H^m_{\text{CE}}(L,M)=Z^m_{\text{CE}}(L,M)/B^m_{\text{CE}}(L,M).    $$
	
	\begin{prop}
		Let $L$ be a Lie algebra over a field of characteristic $p>0$ and suppose that $H_{\text{CE}}^1(L,L)=0$. Then,    $L$ admits  a $p$-map.
	\end{prop}
	\begin{proof}
		Since $H_{\text{CE}}^1(L,L)=0$, every derivation is inner. In particular, if $\{e_1,\cdots,e_n\}$ is a basis of $L$,  the derivation $(\ad_{e_i})^p,~i=1,\cdots,n$ is inner. Therefore, it exists $u_i\in L,~i=1,\cdots,n$ such that $(\ad_{e_i})^p=\ad_{u_i}$. The  conclusion follows from Jacobson's Theorem \ref{jacobson}.
	\end{proof}

	\subsubsection{Restricted cohomology of restricted Lie algebras, $p\geq3$}
	
	Let $(L,[ \cdot , \cdot ],(\cdot )^{[p]})$ be a restricted Lie algebra and $M$ be a  restricted $L$-module. We set $C^0_*(L,M):=C^0_{\text{CE}}(L,M)$ and $C^1_*(L,M):=C^1_{\text{CE}}(L,M)$.
	
	\begin{defi}
		Let $\varphi\in C^2_{\text{CE}}(L,M)$ and let $\omega: L\rightarrow M$ be a map.
		We say that $\omega$ has the $(*)$-property with respect to $\varphi$ if \footnote{In \cite{EF3}, the authors introduced the terminology ``$\w$ is $\varphi$-compatible"}
		\begin{align}
			\omega(\lambda x)&=\lambda^p\omega(x),~\forall \lambda\in \F,~\forall x\in L;\\
			\omega(x+y)&=\omega(x)+\omega(y)+\displaystyle\sum_{\underset{x_1=x,~x_2=y}{x_i\in\{x,y\}}}\frac{1}{\sharp\{x\}}\sum_{k=0}^{p-2}(-1)^kx_p\cdots x_{p-k+1}\varphi([[\cdots[x_1,x_2],x_3]\cdots,x_{p-k-1}],x_{p-k}),
		\end{align} for all $x,y\in L$ and $\sharp\{x\}$ is the number of factors $x_i$ equal to $x$.
		We set 
		$$ C^2_*(L,M)=\left\lbrace (\varphi,\omega),~\varphi\in C^2_{\text{CE}}(L,M),~\omega: L\longrightarrow M \text{ has the $(*)$-property w.r.t } \varphi \right\rbrace.     $$
	\end{defi}
	
	\noindent\textbf{Example.} Let $p=3$ and  $\varphi\in C^2_{\text{CE}}(L,M)$. Then a map $\omega:L\rightarrow M$ has the ($*$)-property with respect to $\varphi$ if and only if
	\begin{align}
		\omega(\lambda x)&=\lambda^3\w(x),~\forall\lambda\in \F,~\forall x\in L;\\
		\w(x+y)&=\w(x)+\w(y)+\varphi([x,y],y)+\frac{1}{2}\varphi([x,y],x)-\frac{1}{2}x\cdot\varphi(x,y)-y\cdot\varphi(x,y),~\forall x,y\in L.\label{pssh}
	\end{align}
	Eq. (\ref{pssh}) can be rewritten as
	$$ \w(x+y)=\w(x)+\w(y)+\varphi([x,y],y)-\varphi([x,y],x)+x\cdot\varphi(x,y)-y\cdot\varphi(x,y),~\forall x,y\in L.$$

	\begin{defi}
		Let $\alpha\in C^3_{\text{CE}}(L,M)$ and $\beta: L\times L\longrightarrow M$ be a map. We say that $\beta$ has the $(**)$-property with respect to $\alpha$ if
		\begin{enumerate}
			\item $\beta(x,y)$ is linear with respect to $x$;
			\item $\beta(x,\lambda y)=\lambda^p\beta(x,y)$;
			\item We have\begin{align*}\beta (x,y_1+y_2)&=\beta(x,y_1)+\beta(x,y_2)\\&-\displaystyle\sum_{\underset{h_1=y_1,~h_2=y_2}{h_i\in\{y_1,y_2\}}}\frac{1}{\sharp\{y_1\}}\sum_{j=0}^{p-2}(-1)^j\\&\times\sum_{k=1}^{j}\binom{j}{k}h_p\cdots h_{p-k-1}\alpha\left( [\cdots[x,h_{p-k}],\cdots,h_{p-j+1}],[\cdots[h_1,h_2],\cdots,h_{p-j-1}],h_{p-j} \right),\end{align*} 
		\end{enumerate}
		for all $\lambda\in \F$, for all $x,y,y_1,y_2\in L$ and with $\sharp\{y_1\}$ the number of factors $h_i$ equal to $y_1$. We set 
		$$ C^3_*(L,M)=\left\lbrace (\alpha,\beta),~\alpha\in C^3_{\text{CE}}(L,M),~\beta: L\times L\longrightarrow M \text{ has the $(**)$-property w.r.t } \alpha \right\rbrace.     $$
		
	\end{defi}
	
	
	
	\noindent An element $\varphi\in C^1_*(L,M)$ induces a map\footnote{In \cite{EF08}, the maps $\ind^1$ and $\ind^2$ are of opposite sign. We use the present convention in order to deal with deformations of restricted morphisms later.}
	\begin{align*}
		\ind^1(\varphi): L&\longrightarrow M\\
		x&\longmapsto -\varphi\bigl(x^{[p]}\bigl)+x^{p-1}\varphi(x). 
	\end{align*}
	
	\noindent An element $(\alpha,\beta)\in C_*^2(L,M)$ induces a map
	\begin{align*}
		\ind^2(\alpha,\beta):L\times L&\longrightarrow M \\
		(x,y)&\longmapsto -\alpha\bigl( x,y^{[p]}\bigl) +\displaystyle\sum_{i+j=p-1}(-1)^iy^i \alpha\bigl( [\cdots[x,\overset{j\text{ terms}}{\overbrace{y],\cdots,y}}],y\bigl)-x\beta(y).
	\end{align*}

	\begin{lem}[\text{\cite{EF08}}]
		The map $\ind^1(\varphi)$ satisfies the $(*)$-property with respect to $d^1_{\text{CE}}\varphi$, and the map $\ind^2(\alpha,\beta)$ satisfies the $(**)$-property with respect to $d_{\text{CE}}^2\alpha$.
	\end{lem}
	
	\begin{defi}
		The restricted differentials are defined as follows:
		\begin{align*}
			d_*^0:~&C_*^0(L,M)\longrightarrow C^1_*(L,M),~d_*^0=d_{\text{CE}}^0;\\
			d_*^1:~&C_*^1(L,M)\longrightarrow C^2_*(L,M),~d_*^1(\varphi)=\left( d_{\text{CE}}^1\varphi,\ind^1(\varphi)\right);\\
			d_*^2:~&C_*^2(L,M)\longrightarrow C^3_*(L,M),~d_*^2(\alpha,\beta)=\left( d_{\text{CE}}^2\alpha,\ind^2(\alpha,\beta)\right).
		\end{align*}
	\end{defi}
	If $ m\in \{1,2\}$, we have $d_*^m\circ d^{m-1}_*=0$. We denote by $Z_*^m(L,M)=\Ker(d^m_*)$ the restricted $m$-cocycles and $B_*^m(L,M)=\im(d^{m-1}_*)$ the restricted $m$-coboundaries. We denote the \textit{restricted cohomology groups} by
	$$ H_*^m(L,M):=Z_*^m(L,M)/B^m_*(L,M).        $$
	
	\noindent\textbf{Remark:} $H^0_*(L,M)=H^0_{\text{CE}}(L,M)$.

	\section{Deformation theory of restricted Lie algebras, $p\geq3$}\label{Deformation theory of restricted Lie algebras}
	
	Let $\F$ be a field of characteristic $p\geq 3$. In \cite{EF08}, Evans and Fuchs have sketched a deformation theory of restricted Lie algebras. In this Section, we investigate formal restricted deformations of restricted Lie algebras as well as equivalence of such deformations and restricted obstructions. We also introduce the notion of deformation of restricted morphisms and an adapted (partial) deformation cohomology. 
	
	\subsection{Restricted formal deformations}
	
	Let $\bigl(L,[\cdot,\cdot],(\cdot)^{[p]}\bigl)$ be a restricted Lie algebra. We aim to study  deformations of  both the Lie bracket and the $p$-map.
	
	\begin{defi}\label{defidefop}
		A \textit{formal deformation} of $\bigl(L,[\cdot,\cdot],(\cdot)^{[p]}\bigl)$ is given by two maps

\begin{equation*}
\begin{array}{lllll}
m_t:&L\times L\longrightarrow L[[t]]& \text{and }\quad & \w_t:&L\longrightarrow L[[t]] \\[2mm]
&(x,y)\longmapsto \displaystyle\sum_{i\geq 0}t^i m_i(x,y)&\quad  &&x\longmapsto \displaystyle\sum_{j\geq 0}t^j\omega_j(x), \\[2mm]
 \end{array}
\end{equation*}
where $m_0=[\cdot,\cdot]$, $\w_0=(\cdot)^{[p]}$, and $(m_i,\w_i)\in C^2_*(L,L)$.
		Moreover, the two following conditions must be satisfied, for all $x,y,z\in L$:
		
		\begin{equation}\label{jacomulti}
			m_t(x,m_t(y,z))_t+m_t(y,m_t(z,x))+m_t(z,m_t(x,y))=0;
		\end{equation}
		\begin{equation}\label{pmapmulti}
			m_t\left(x,\w_t(y)\right)_t=m_t( m_t(\cdots m_t(x,  \overset{p\text{ terms}}{\overbrace{y),y),\cdots,y}}).
		\end{equation}
	\end{defi}

	\noindent\textbf{Remarks.}
		\hspace{0.2cm}
		\begin{enumerate}
			\item The map $m_t$ extends to $L[[t]]$ by $\F[[t]]$-linearity.
			\item The map $\w_t$ extends to $L[[t]]$ by $p$-homogeneity and by using the formula $$\w_t(x+ty)=\w_t(x)+t^p\w_t(y)+\displaystyle\sum_{k=1}^{p-1}\tilde{s}(x,ty),$$ with $k\tilde{s}(x,ty)$ being the coefficient of $Z^{p-1}$ in the formal expression $m_t(xZ+ty,x)$.
		\end{enumerate}

	\noindent\textbf{Notation.} In the sequel, if $f$ is a map with two variables and $n\geq2$, we will use the notation $$f(x_1,\cdots,x_n):=f(f(\cdots f(x_1,x_2),x_3),\cdots,x_n).$$
    In particular, we have
    $$[x_1,\cdots,x_n]:=[[\cdots [x_1,x_2],x_3],\cdots,x_n].$$
	\begin{lem}\label{cstar}
		Let $\left(m_t, \w_t \right)$ be a restricted deformation of $\bigl(L,[\cdot,\cdot],(\cdot)^{[p]}\bigl)$. Then $\w_1$ has the $(*)$-property with respect to $m_1$.
	\end{lem}
	
	\begin{proof}
		Let $\lambda\in \F,~x\in L$.
		\begin{enumerate}
			\item[$\bullet$] $\w_t(\lambda x)=\lambda^p \w_t(x)\Leftrightarrow (\lambda x)^{[p]}+t\w_1(\lambda x)+\displaystyle\sum_{i\geq 2}t^i\w_i(\lambda x)=\lambda^p\Bigl(x^{[p]}+t\w_1(x)+ \displaystyle\sum_{i\geq 2}t^i\w_i(x)\Bigl). $ By collecting the coefficients of $t$ in the above equation, we obtain $\w_1(\lambda x)=\lambda^p\w_1(x)$.
			
			\item[$\bullet$] The following computations are made modulo $t^2$. We have
			\begin{equation}
				m_t(x_1,\cdots,x_p)=[x_1,\cdots,x_p]+t\sum_{k=0}^{p-2}\left[m_1([x_1,\cdots,x_{p-k-1}],x_{p-k}),x_{p-k+1},\cdots,x_p \right]  . 
			\end{equation}
			
			We denote by $(\maltese)$ the three conditions $ \bigl( x_i\in\{x,y\},~x_1=x,~x_2=y\bigl) $.   Using the definition of $\w_t$ and Eq.(\ref{si}), we have
			\begin{align*}
				\hspace{-1cm}\w_t(x+y)&=\w_t(x)+\w_t(y)+\sum_{(\maltese)} \frac{1}{\sharp\{x\}}m_t(x_1,x_2,\cdots,x_p)\\
				&=x^{[p]}+y^{[p]}+\sum_{(\maltese)}\frac{1}{\sharp\{x\}}[x_1,x_2,\cdots,x_p ]\\&\hspace{1cm}+t\sum_{(\maltese)}\sum_{k=0}^{p-2}\left[m_1([x_1,\cdots,x_{p-k-1}],x_{p-k}),x_{p-k+1},\cdots,x_p\right]             \mod(t^2). 
			\end{align*}
			
			We also have $\w_t(x+y)=(x+y)^{[p]}+t\w_1(x+y)\mod(t^2).$ If we compare the coefficients in the above expressions, we obtain
			\begin{align}
				(x+y)^{[p]}&=x^{[p]}+y^{[p]}+\sum_{(\maltese)}\frac{1}{\sharp\{x\}}[x_1,\cdots,x_p] ;\\
				\nonumber
				\w_1(x+y)&=\w_1(x)+\w_1(y)+\sum_{(\maltese)} \sum_{k=0}^{p-2}\left[m_1([x_1,\cdots,x_{p-k-1}],x_{p-k}),x_{p-k+1},\cdots,x_p\right]\\
				&=\w_1(x)+\w_1(y)+\sum_{(\maltese)} \sum_{k=0}^{p-2}(-1)^{k}\left(x_p x_{p-1}\cdots x_{p-k-1}m_1([x_1,\cdots,x_{p-k-1}],x_{p-k})\right).\end{align}
		\end{enumerate}
		
		We conclude that $\w_1$ has the $(*)$-property with respect to $m_1$.
		
	\end{proof}
	
	\noindent\textbf{Example: The case $p=3$.} Let $\left(L,[\cdot,\cdot],(\cdot)^{[p]}\right)$ be a restricted Lie algebra. Consider an infinitesimal deformation $(m_t,\w_t)$ of $\left(L,[\cdot,\cdot],(\cdot)^{[p]}\right)$ given by $m_t=[\cdot,\cdot]+tm_1,~\w_t=(\cdot)^{[p]}+t\w_1$, and let $x,y\in L$.
	\begin{align*}
		\w_t(x+y)-\w_t(x)-\w_t(y)=&\sum_{\underset{x_{1}=x,~x_2=y}{x_k\in\{x,y\}}}\frac{1}{\sharp\{x\}}m_t(m_t(x_1,x_2),x_3)\\
		=&~2m_t(m_t(x,y),x)+m_t(m_t(x,y),y)\\
		=&~2\bigl([[x,y],x]+tm_1([x,y],x)+t[m_1(x,y),x] \bigl)\\&~+[[x,y],y]+tm_1([x,y],y)+t[m_1(x,y),y].\\
	\end{align*}
	Collecting the coefficients of $t$ on both sides, we obtain
	$$\w_1(x+y)-\w_1(x)-\w_1(y)=2m_1([x,y],x)+2[m_1(x,y),x]+m_1([x,y],y)+[m_1(x,y),y],$$ which exactly means that $\w_1$ satisfies the $(*)$-property with respect to $m_1$.
	
	\begin{thm}\label{restdefo}
		Let $\left(m_t, \w_t \right)$ be a restricted deformation of $\left(L,[\cdot,\cdot],(\cdot)^{[p]}\right)$. Then $(m_1,\w_1)$ is a 2-cocyle of the restricted cohomology.
	\end{thm}
	
	\begin{proof}
		By Lemma \ref{cstar}, $(m_1,\w_1)\in C^2_*(L,L)$.
		By the ordinary deformation theory, we already have $m_1\in Z^2_{\text{CE}}(L,L)$. It remains to show that $\ind^2(m_1,\w_1)$ vanishes. We expand the equation 
		\begin{equation}
			m_t\left(x,\w_t(y) \right)=m_t( m_t(\cdots m_t(x,  \overset{p\text{ terms}}{\overbrace{y),y),\cdots,y}}).
		\end{equation}
		On one hand,
		\begin{align*}
			m_t\left(x,\w_t(y) \right)&=[x,\w_t(y)]+t m_1(x,\w_t(y))+\sum_{i\geq 2}t^im_i(x,\w_t(y))\\
			&=[x,y^{[p]}]+t m(x,\w_1(y))+t m_1(x,y^{[p]}) \mod(t^2).
		\end{align*}
		On the other hand,
		\begin{align*}
			m_t( m_t(\cdots m_t(x,  \overset{p\text{ terms}}{\overbrace{y),y),\cdots,y}})&
			=\displaystyle \sum_{i_p}\cdots\sum_{i_1}t^{i_1+\cdots+i_p}m_{i_p}\left(m_{i(p-1)}(\cdots(m_{i_1}(x,y),y),\cdots,y),y \right)\\
			&=[x,y, \cdots ,y]+t\sum_{\underset{\sharp\left\lbrace k,~i_k=1 \right\rbrace=1 }{i_k=0 \text{ or } 1}}m_{i_p}\left(\cdots(m_{i_1}(x,y),y),\cdots,y),y \right) \mod(t^2)\\
			&=[x,y, \cdots ,y]+t\sum_{i+j=p-1}\Bigl[m_1([x,\overset{j}{\overbrace{y,\cdots,y}}],y),\overset{i}{\overbrace{y,\cdots,y}}   \Bigl] \mod(t^2)\\
			&=[x,y, \cdots ,y]+t\sum_{i+j=p-1}(-1)^iy^im_1([x,\overset{j}{\overbrace{y,\cdots,y}}],y) \mod(t^2).\\
		\end{align*}
		We finally obtain the equation
		\begin{equation}
			[x,y^{[p]}]+t\bigl([x,\w_1(y)]+m_1(x,y^{[p]})\bigl) =[x,y, \cdots ,y]+t\sum_{i+j=p-1}(-1)^i y^i m_1([x,\overset{j}{\overbrace{y,\cdots,y}}],y).
		\end{equation}
		By collecting the coefficients of $t^0$ and $t$, we recover the usual identity $[x,y^{[p]}]=[x,y,\cdots,y]$ and obtain a new identity
		\begin{equation}  [x,\w_1(y)]+m_1(x,y^{[p]})=\sum_{i+j=p-1}(-1)^iy^im_1([x,\overset{j}{\overbrace{y,\cdots,y}}],y). 
		\end{equation}
		Therefore, we have $\ind^2(m_1,\w_1)=0$. 
	\end{proof}
	
	\subsection{Equivalence of restricted formal deformations}
	
	Let $\phi_t:L[[t]]\longrightarrow L[[t]]$ be a formal automorphism defined on $L$ by 
	$$ \phi_t(x)=\sum_{i\geq 0}t^i\phi_i(x),~\phi_i:L\longrightarrow L~\F\text{-linear },~\phi_0=\text{id},  $$ and then extended by $\F[[t]]$-linearity.
	
	\begin{defi}
		Let $\left(m_t, \w_t \right)$ and $\left(m'_t, \w'_t \right)$ be two formal deformations of $\left(L,[\cdot,\cdot],(\cdot)^{[p]}\right)$. They are called \textit{equivalent} if there is a formal automorphism $\phi_t$ such that
		\begin{equation}
			m_t'(\phi_t(x),\phi_t(y))=\phi_t\left(m_t(x,y) \right)
		\end{equation}
		and
		\begin{equation}	
			\w_t'\left(\phi_t(x)\right)=\phi_t\left(\w_t(x) \right).
		\end{equation}
	\end{defi}
	
	\begin{lem}\label{equivalem}
		Let $\left(m_t, \w_t \right)$ and $\left(m_t', \w'_t \right)$ be two equivalent formal deformations of $\left(L,[\cdot,\cdot],(\cdot)^{[p]}\right)$. Then there exists $\psi: L\rightarrow L$ such that, for all $ x,y\in L$,
		\begin{equation}
			m_1'(x,y)-m_1(x,y)=\psi\left([x,y]\right)-[\left(x,\psi(y)\right)]-[\left(\psi(x),y \right)]
		\end{equation}
		and
		\begin{equation}
			\w_1'(x)-\w_1(x)=\psi(x^{[p]})-[\psi(x), \overset{p-1}{\overbrace{x,\cdots,x}}].
		\end{equation}  
		If the equivalence is given by $\phi_t=\displaystyle \sum_{i\geq 0} t^i \phi_i$, then $\psi=\phi_1$.
	\end{lem}
	
	\begin{proof}
		Let $x,y\in L$. Since
		$$m_t'\bigl(\phi_t(x),\phi_t(y)\bigl)=\phi_t\left(m_t(x,y) \right),$$ we have $$\sum_{k\geq 0}t^km_t'(\phi_t(x),\phi_t(y))=\sum_{k\geq 0}t^k\phi_t\left(m_t(x,y) \right).     $$
		We deduce that
		$$
		[\phi_t(x),\phi_t(y)]+tm_1'(\phi_t(x),\phi_t(y)
		=\phi_t([x,y])+t\phi_t(m_1(x,y))\mod(t^2).$$
		Therefore,
		\begin{align*}
			\Bigl[\sum_{i\geq 0}t^i\phi_i(y),\sum_{j\geq 0}t^j\phi_j(y) \Bigl]&+tm_1'\Bigl(\sum_{i\geq 0}t^i\phi_i(y),\sum_{j\geq 0}t^j\phi_j(y) \Bigl)\\&=\sum_{i\geq 0}t^i\phi_i([x,y])+t\sum_{j\geq 0}t^j\phi_j(m_1(x,y))\mod(t^2).
		\end{align*}
		Hence, 
		$$[x,y]+t\Bigl( [x,\phi_1(y)]+[\phi_1(x),y]+m_1'(x,y) \Bigl)= [x,y]+t\bigl(\phi_1([x,y])+m_1(x,y) \bigl)\mod(t^2). $$
		By collecting the coefficients of $t$, we obtain the first identity. Then, expanding the relation $\phi_t\left(\w_t(x) \right)=\w'_t\left(\phi_t(x)\right)$ yields
		$$\phi_t\Bigl(\sum_{i>0}t^i\w_i(x) \Bigl)=\w_t'\bigl(x+t\phi_1(x) \bigl)\mod (t^2) $$ and it follows that
    \begin{equation}\label{panicattack}x^{[p]}+t\bigl(\w_1(x)+\phi_1(\w(x)) \bigl)=\w_t'\left(x+t\phi_1(x) \right) \mod(t^2).\end{equation}
		We compute the right hand side of Eq.\eqref{panicattack}. We denote once again by $(\maltese)$ the conditions $( x_i\in\{x,y\},~x_1=x,~x_2=y )$, where we temporarily denote $y:=t\phi_1(x)$. Then, we have
		\begin{align*}
			\w_t'\left(x+t\phi_1(x) \right)&=\w_t'(x)+\w_t'(y)+\sum_{(\maltese)}\frac{1}{\sharp\{x\}}m_t(x_1,\cdots,x_p)\\
			&=\w_t'(x)+\w_t'(y)+\sum_{(\maltese)}\frac{1}{\sharp\{x\}}[x_1,\cdots,x_p]\mod(t^2)\\
			&=\w_t'(x)+\w_t'(t\phi_1(x))+\frac{1}{p-1}[x,t\phi_1(x),x,x,\cdots,x]\mod(t^2)\\
			&=\w_t'(x)-t[x,\phi_1(x),x,x,\cdots,x]\mod(t^2)\\
			&=x^{[p]}+t\bigl(\w_1'(x)-[x,\phi_1(x),x,x,\cdots,x] \bigl) \mod(t^2).\\
		\end{align*}
		Therefore, Eq.\eqref{panicattack} gives
		$$x^{[p]}+t\bigl(\w_1(x)+\phi_1(x^{[p]}) \bigl)=x^{[p]}+t\bigl(\w_1'(x)-[x,\phi_1(x),x,x,\cdots,x] \bigl) \mod(t^2).$$
		Collecting the coefficients of $t$ in the previous equation, we obtain
		$$ \w_1(x)-\w_1'(x)=[\phi_1(x),\overset{p-1}{\overbrace{x,\cdots,x}}]-\phi_1(x^{[p]}),$$ which is the desired identity.
	\end{proof}
	
	\begin{rmq}
		Lemma \ref{equivalem} allows to recover the definitions given in \cite{EF08} and \cite{ET00} in the case of infinitesimal deformations of restricted Lie algebras.
	\end{rmq}

	\begin{thm}\label{restequiva}
		Let $\left(m_t, \w_t \right)$ and $\left(m_t', \w'_t \right)$ be two equivalent formal deformations of $\left(L,[\cdot,\cdot],(\cdot)^{[p]}\right)$. Then, their infinitesimal elements are in the same cohomology class.
	\end{thm}
	\begin{proof}
		Let $\left(m_t, \w_t \right)$ and $\left(m_t', \w'_t \right)$ be two equivalent formal deformations via $\phi=\displaystyle\sum_{i\geq 0}t^i\phi_i$. Let $x,y\in L$. By Lemma \ref{equivalem}, we have
		\begin{equation}\label{CEeq}
			m_1'(x,y)-m_1(x,y)=\phi_1\left([x,y]\right)-[\left(x,\phi_1(y)\right)]-[\left(\phi_1(x),y \right)]\text{ and }
		\end{equation}
		\begin{equation}\label{pmapeq}
			\w_1(x)-\w_1'(x)=[\phi_1(x), \overset{p-1}{\overbrace{x,\cdots,x}}]-\psi(x^{[p]}).     
		\end{equation}
		The map $\phi_1$ belongs to $C^1_{\text{CE}}(L,L)$. We have
		\begin{align*}
			d^1_{\text{CE}}(\phi_1)(x,y)
			&=-\left(\phi_1([x,y])-[x,\phi_1(y)]-[\phi_1(x),y] \right).
		\end{align*}
		Using Equation (\ref{CEeq}), we deduce that $m_1'(x,y)-m_1(x,y)=-d^1\phi_1(x,y)$, therefore $\left( m_1'-m_1\right) \in B^2_{\text{CE}}(L,L)$. Moreover, we have
		\begin{align*}
			\ind^1(\phi_1)&=-\phi_1(x^{[p]})+[\phi_1(x),x,x,\cdots,x]\\
			&=\w_1(x)-\w_1'(x).  
		\end{align*}	
		Finally, $\left(m_1'-m_1,~\w_1'-\w_1\right)=-d_*^1\phi_1\in B^2_*(L,L). $ 
		We conclude that $(m_1,\w_1)$ and	$(m_1',\w_1')$ are equal up to  a coboundary.
		
	\end{proof}
	
	\begin{defi}
		Let $(m_t,\w_t)$ be a restricted deformation of $\left(L,[\cdot,\cdot],(\cdot)^{[p]}\right)$.
		The deformation is called \textit{trivial} if there is a formal automorphism $\phi_t$ such that
		\begin{equation}\label{eqtriv1}
			\phi_t\left([x,y]\right)=m_t(\phi_t(x),\phi_t(y))
		\end{equation}
		and
		\begin{equation}\label{eqtriv2}
	   \phi_t\left(\w_t(x)\right)=\phi_t(x)^{[p]}.
		\end{equation}
	\end{defi}
    Expanding Eq.(\ref{eqtriv1}) mod $(t^2)$ yields $m_1=-d_{\text{CE}}^1(\phi_1)$. Now, we focus on the right-hand side of Eq.(\ref{eqtriv2}). The following computations are made mod $t^2$, for all $x\in L$:
	\begin{align*}
		\phi(x)^{[p]}&= \Bigl( \sum_it^i\phi_i(x) \Bigl)^{[p]}\\
		&=x^{[p]}+\left(t\phi_1(x)\right)^{[p]}+\displaystyle\sum_{\underset{x_{p-1}=t\phi_1(x),~x_p=x}{x_i\in\{x,~t\phi_1(x)\}}}\frac{1}{\sharp\{x\}} [x_1,[\cdots,[x_{p-1},x_p]\cdots].   \\
		&=x^{[p]}+\left(t\phi_1(x)\right)^{[p]}+\frac{1}{p-1}[x,[x,\cdots [t\phi_1(x),x]\cdots]\\
		&=x^{[p]}+\left(t\phi_1(x)\right)^{[p]}+t\ad_x^{p-1}\circ\phi_1(x).
	\end{align*}
	For the left-hand side, we have (mod $t^2$):
	\begin{align*}
		\phi\left(\w_t(x)\right)&=\sum_it^i\phi_i\left(\w_t(x)\right)\\
		&=\sum_i\sum_jt^{i+j}\phi_i\left(\omega_j(x)\right)\\
		&=x^{[p]}+t\bigl( \omega_1(x)+\phi_1(x^{[p]})\bigl).
	\end{align*}
	We deduce that 
	\begin{equation}
		\omega_1(x)+\phi_1(x^{[p]})=\ad_x^{p-1}\circ\phi_1(x),
	\end{equation}
	which can be rewritten
	\begin{equation}
		\omega_1(x)=-\phi_1(x^{[p]})+\ad_x^{p-1}\circ\phi_1(x)=-\ind^1(\phi_1)(x).
	\end{equation}
	Hence, we proved the following result.
	
	\begin{prop}\label{trivrest}
		The deformation $\left( [\cdot,\cdot]+tm_1, (\cdot)^{[p]}+t\w_1 \right)$ is trivial  if and only if  $(m_1,\w_1)$  is a restricted coboundary.
	\end{prop}

	\subsection{Obstructions}\label{sectionobstructions}
	Let $\left(L,[\cdot,\cdot],(\cdot)^{[p]}\right)$ be a restricted Lie algebra and $n\geq 1$. A deformation of $\left(L,[\cdot,\cdot],(\cdot)^{[p]}\right)$ is said of order $n$ if it is of the form

	$$m^n_t=\sum_{i=0}^{n}t^im_i;~~~\w_t^n=\sum_{i=0}^{n}t^i\w_i.$$
	
	\begin{defi}
		Let $\left( m_t^n,\w_t^n \right) $ be a $n$-order deformation of  $\left(L,[\cdot,\cdot],(\cdot)^{[p]}\right)$. We define for all $x,y,z\in L,$ the maps
		\begin{align*}
			\obs^{(1)}_{n+1}(x,y,z)&=-\sum_{i=1}^{n}\left(m_i(x,m_{n+1-i}(y,z))+m_i(y,m_{n+1-i}(z,x))+m_i(z,m_{n+1-i}(x,y)) \right);\\
			\obs^{(2)}_{n+1}(x,y)~~&=\sum_{i=1}^{n}m_i(x,\w_{n+1-i}(y))-\sum_{\underset{i_1+\cdots+i_p=n+1}{0\leq i_k\leq n}}m_{i_p}\left(m_{i_{p-1}}(\cdots(m_{i_1}(x,y),y),\cdots,y),y \right).\\
		\end{align*}		
	\end{defi}

	\begin{prop}
		Let $\left( m_t^n,\w_t^n \right) $ be a $n$-order deformation of $\bigl(L,[\cdot,\cdot],(\cdot)^{[p]}\bigl)$. Let $\left(m_{n+1},~\w_{n+1} \right)\in C^2_*(L,L)$. Suppose that $\left( m^n_t+t^{n+1}m_{n+1},~  \w^n_t+t^{n+1}\w_{n+1} \right) $ is a $(n+1)$-order deformation of  $\left(L,[\cdot,\cdot],(\cdot)^{[p]}\right)$. Then  	
		$$	\left( \obs^{(1)}_{n+1}, \obs^{(2)}_{n+1}\right)=d_*^2\left(m_{n+1},~\w_{n+1} \right).        $$
	\end{prop}

	\begin{proof}
		Let $\left( m^n_t+t^{n+1}m_{n+1},~  \w^n_t+t^{n+1}\w_{n+1} \right)$ be a $(n+1)$-order deformation of  $\left(L,[\cdot,\cdot],(\cdot)^{[p]}\right)$. The deformed operation $m_t^n+t^{n+1}m_{n+1}$ should satisfy the Jacobi identity, that is,
		$$\sum_{k=0}^{n+1}t^{k}\sum_{i=0}^{k}\left(m_i(x,m_{k-i}(y,z))+m_i(y,m_{k-i}(z,x))+m_i(z,m_{k-i}(x,y)) \right)=0~(\text{mod }t^{n+2}),~\forall x,y,z\in L.$$ 
		By collecting the coefficients of $t^{n+1}$, we obtain
		$$ \sum_{i=0}^{n+1}\left(m_i(x,m_{n+1-i}(y,z))+m_i(y,m_{n+1-i}(z,x))+m_i(z,m_{n+1-i}(x,y)) \right)=0,    $$
		which can be written
		\begin{align}
			&[x,m_{n+1}(y,z)]+[y,m_{n+1}(z,x)]+[z,m_{n+1}(x,y)]\nonumber\\ +&m_{n+1}(x,[y,z])+m_{n+1}(y,[z,x])+m_{n+1}(z,[x,y])\\+& \sum_{i=1}^{n}\left(m_i(x,m_{n+1-i}(y,z))+m_i(y,m_{n+1-i}(z,x))+m_i(z,m_{n+1-i}(x,y)) \right)=0.
		\end{align}  
    Hence
		$$-d^2_{\text{CE}}m_{n+1}(x,y,z)=\sum_{i=1}^{n}\left(m_i(x,m_{n+1-i}(y,z))+m_i(y,m_{n+1-i}(z,x))+m_i(z,m_{n+1-i}(x,y)) \right).$$ Therefore, we have $d^2_{\text{CE}}m_{n+1}(x,y,z)=\obs^{(1)}_{n+1}(x,y,z)$. 
		
		\vspace{0.5cm}
		
		Denote by $\w_t^n+t^{n+1}\w_{n+1}:=\w_t^{n+1}$. Suppose that $\w_t^{n+1}$ is a deformation of order $(n+1)$, then we have
		\begin{equation}
			m_t^{n+1}\left(x,\w_t^{n+1}(y) \right)=m_t^{n+1}( m_t^{n+1}(\cdots m_t^{n+1}(x,  \overset{p\text{ terms}}{\overbrace{y),y),\cdots,y}}),~\forall x,y\in L.
		\end{equation}	
		By expanding and collecting the coefficients of $t^{n+1}$, we obtain	
		$$\sum_{i=0}^{n+1}m_i(x,\w_{n+1-i}(y))=\sum_{\underset{i_1+\cdots+i_p=n+1}{0\leq i_k\leq n+1}}m_{i_p}\left(m_{i(p-1)}(\cdots(m_{i_1}(x,y),y),\cdots,y),y \right),    $$	
		which can be rewritten	
		\begin{align*}
			&[x,\w_{n+1}(y)]+m_{n+1}(x,\w(y))-\sum_{i+j=p-1}\bigl[ m_{n+1}([x,\overset{j}{\overbrace{y,\cdots,y}}],y),\overset{i}{\overbrace{y,\cdots,y}}\bigl]\\
			=&\sum_{\underset{i_1+\cdots+i_p=n+1}{0\leq i_k\leq n}}m_{i_p}\left(m_{i(p-1)}(\cdots(m_{i_1}(x,y),y),\cdots,y),y \right)-\sum_{i=1}^{n}m_i(x,\w_{n+1-i}(y)).\\
		\end{align*}
		Using the definition of $\ind^2(m_{n+1},\w_{n+1})$, we finally obtain		
		$$d^2_*\left(m_{n+1},\w_{n+1} \right) =\left(d^2_{\text{CE}}m_{n+1},\ind^2(m_{n+1},\w_{n+1})\right)= \left( \obs^{(1)}_{n+1}, \obs^{(2)}_{n+1}\right).     $$

	\end{proof}
	In the usual algebraic deformation theory, the obstructions are 3-cocycles. Since we do not yet have a definition for $d^3_*$, we cannot assert that   $\bigl( \obs^{(1)}_{n+1}, \obs^{(2)}_{n+1}\bigl)$ is a $3$-cocycle of the deformation cohomology.
	
	\subsection{Deformation of restricted morphisms}\label{pmorphdefo}

 In this section, we investigate restricted deformations of restricted morphisms. We introduce new restricted cohomology spaces controlling those deformations. Our approach follows \cite{Ma07}.

		\subsubsection{Deformation cohomology of restricted morphisms}

	Let $(L,[ \cdot , \cdot ]_L,(\cdot )^{[p]_L})$ and $(M,[ \cdot , \cdot ]_M,(\cdot )^{[p]_M})$ be restricted Lie algebras and let $\varphi:L\rightarrow M$ be a restricted morphism. The restricted Lie algebra $M$ has a restricted $L$-module structure given by $x\cdot m:=[\varphi(x),m]_M$. Let $n\geq1$. We define 
	$$\fC^n_{\ce}(\varphi,\varphi):=C^n_{\ce}(L,L)\times  C^n_{\ce}(M,M)\times C^{n-1}_{\ce}(L,M),$$
	and $\fC^0_{\ce}(\varphi,\varphi):=0$. We define differential maps
	\begin{align*}
		\fd^n_{\ce}:\fC^n_{\ce}(\varphi,\varphi)&\rightarrow\fC^{n+1}_{\ce}(\varphi,\varphi)\\
		(\mu,\nu,\theta)&\mapsto \bigl(d^n_{\ce}\mu,d^n_{\ce}\nu,\alpha_{\mu,\nu}(\theta)\bigl),
	\end{align*}
	with $\alpha_{\mu,\nu}(\theta):=\varphi\circ\mu-\nu\circ(\varphi^{\otimes n})-d^{n-1}_{\ce}\theta$.\\We denote $\fZ^n_{\ce}(\varphi,\varphi):=\Ker(\fd^n_{\ce})$ and  $\fB^n_{\ce}(\varphi,\varphi):=\im(\fd^{n-1}_{\ce})$.
	\begin{prop}[see \cite{NR67.2, Ma07}]\label{propcoce}
		We have $\fd^{n+1}_{\ce}\circ\fd^n_{\ce}=0$ and the quotient spaces $$\fH^n_{\ce}(\varphi,\varphi):=\fZ^n_{\ce}(\varphi,\varphi)/\fB^n_{\ce}(\varphi,\varphi)$$ are well defined.
	\end{prop}
	Furthermore, we define the restricted cochain spaces
	\begin{align*}
		\fC^0_{*}(\varphi,\varphi)&:=0;\\
		\fC^1_{*}(\varphi,\varphi)&:=C^1_{*}(L,L)\times  C^1_{*}(M,M)\times C^0_{*}(L,M);\\
		\fC^2_{*}(\varphi,\varphi)&:=C^2_{*}(L,L)\times  C^2_{*}(M,M)\times C^1_{*}(L,M);\\
		\fC^3_{*}(\varphi,\varphi)&:=C^3_{*}(L,L)\times  C^3_{*}(M,M)\times \widetilde{C}^2_{*}(L,M),
	\end{align*}
	where $\widetilde{C}^2_{*}(L,M):=\left\lbrace (\varphi,\w),~\varphi\in C^2_{\text{CE}}(L,M),~\w: L\rightarrow M,~\w \text{ is $p$-homogeneous}\right\rbrace.$ We define the restricted differentials 		
	\begin{align*}
		\fd_*^0:~&\fC_*^0(L,M)\rightarrow \fC^1_*(L,M),~\fd_*^0:=\fd_{\text{CE}}^0;\\
		\fd_*^1:~&\fC_*^1(L,M)\rightarrow \fC^2_*(L,M),~\fd_*^1(\gamma,\tau,m):=\begin{pmatrix} \bigl(d_{\text{CE}}^1\gamma,\ind^1(\gamma)\bigl)\\\bigl(d_{\text{CE}}^1\tau,\ind^1(\tau)\bigl)\\\alpha_{\gamma,\tau}(m)\end{pmatrix};\\
		\fd_*^2:~&\fC_*^2(L,M)\rightarrow \fC^3_*(L,M),~\fd_*^2\bigl((\mu,\omega),(\nu,\epsilon),\theta\bigl):=\begin{pmatrix} \bigl(d_{\text{CE}}^2\mu,\ind^2(\mu,\w)\bigl)\\\bigl(d_{\text{CE}}^2\nu,\ind^2(\nu,\epsilon)\bigl)\\\bigl(\alpha_{\mu,\nu}(\theta),\beta_{\w,\epsilon}(\theta)\bigl)\end{pmatrix},
	\end{align*}
	with $\beta_{\w,\epsilon}(\theta)(x):=\theta(x^{[p]})+\varphi(\w(x))-\epsilon(\varphi(x))-x^{p-1}\cdot\theta(x),~\forall x\in L$. \\ We denote $\mathfrak{Z}^n_{*}(\varphi,\varphi):=\Ker(\fd^n_{*})$ and $\mathfrak{B}^n_{*}(\varphi,\varphi):=\im(\fd^{n-1}_{*})$, for $n=1,2$.
	
	\begin{prop}
		We have $\fd^{2}_{*}\circ\fd^1_{*}=0$ and $\fd^{1}_{*}\circ\fd^0_{*}=0$. Therefore, the quotient spaces $\fH^n_{*}(\varphi,\varphi):=\fZ^n_{*}(\varphi,\varphi)/\fB^n_{*}(\varphi,\varphi)$ are well defined, for $n\in\{1,2\}$.
	\end{prop}
	
	\begin{proof}
		Let $x\in L,m\in M,\gamma\in C^1_{*}(L,L)$ and $\tau \in C^1_{*}(M,M)$. Consider the map $\theta: L\rightarrow M$ given by $$\theta(x)=\varphi\circ\gamma(x)-\tau(\varphi(x))+\ad_m\circ\varphi(x).$$Since
		$$-x^{p-1}\ad_m\circ\varphi(x)=\ad^p_{\varphi(x)}(m)=[\varphi(x)^{[p]},m]_M=-[m,\varphi(x^{[p]})] $$ 
		and 
		$$-x^{p-1}\varphi\circ\gamma(x)=-\ad_{\varphi(x)}^{p-1}\bigl(\varphi\circ\gamma(x)\bigl)=-\varphi\bigl(\ad_x^{p-1}\circ\gamma(x)\bigl).  $$	 We have\begin{align*}
			\beta_{\ind^1\gamma,\ind^1\tau}(\theta)(x)&=\varphi\circ\gamma(x^{[p]})-\tau(\varphi(x^{[p]}))+\ad_m\circ\varphi(x^{[p]})\\
			&-\varphi\circ\gamma(x^{[p]})+\varphi\bigl(\ad^{p-1}_m\circ\gamma(x)\bigl)\\
			&+\tau\circ\varphi(x)^{[p]}-x^{p-1}\tau\circ\varphi(x)\\
			&-x^{p-1}\Bigl(\varphi\circ\gamma(x)-\tau\circ\varphi(x)+\ad_m\circ\varphi(x)\Bigl)\\
			&=0.
		\end{align*}
		
		This result together with Proposition \ref{propcoce} and properties of $\ind^1$ imply that $\fd^{2}_{*}\circ\fd^1_{*}=0$.
	\end{proof}

		\subsubsection{Deformations of restricted morphisms}\label{242}

	Let $\bigl(L,[ \cdot , \cdot ]_L,(\cdot )^{[p]_L}\bigl)$ and $\bigl(M,[ \cdot , \cdot ]_M,(\cdot )^{[p]_M}\bigl)$ be restricted Lie algebras and  $\varphi:L\rightarrow M$ be a restricted morphism. We recall that the restricted Lie algebra $M$ has a restricted $L$-module structure given for all $x\in L,m\in M$ by $x\cdot m:=[\varphi(x),m]_M$.
	Let $(\mu_t,\w_t)$ (resp. $(\nu_t,\epsilon_t)$) be a restricted deformation of $L$ (resp. $M$). A restricted deformation of $\varphi$ is a restricted morphism $\varphi_t:\bigl(L[[t]],\mu_t,\w_t\bigl)\rightarrow \bigl(M[[t]],\nu_t,\epsilon_t\bigl)$ given by $$\varphi_t(x):=\displaystyle\sum_{i\geq0}t^i\varphi_i(x),~\varphi_i:L\rightarrow M \text{ linear maps, }\forall x\in L.$$ 
Since $\varphi_t$ is a restricted morphism, it must satisfy the following conditions, for all $x,y\in L$:
\begin{align}
    \varphi_t\circ \mu_t(x,y)&=\nu_t(\varphi_t(x),\varphi_t(y));\label{eqdefmorph1}\\
    \varphi_t\circ\w_t(x)&=\epsilon_t\circ\varphi_t(x).\label{eqdefmorph2}
\end{align}
Expanding Eq.\eqref{eqdefmorph1} modulo $t^2$ implies that $(\mu_1,\nu_1,\varphi_1)\in \fZ^2_{\ce}(\varphi,\varphi)$, see \cite{NR67.2}. Let $x\in L$. Expanding Eq.\eqref{eqdefmorph2} modulo $t^2$ leads to
\begin{align*}
    0&=\varphi_t\bigl(x^{[p]}+t\w_1(x)\bigl)-\epsilon_t\bigl(\varphi(x)+t\varphi_1(x)\bigl)\\
    &=\varphi(x^{[p]})+t\bigl(\varphi(\w_1(x))+\varphi_1(x^{[p]}\bigl)-\varphi(x)^{[p]}-t\epsilon_1(\varphi(x))-\sum_{i=1}^{p-1}s_i\bigl(\varphi(x),t\varphi_1(x)\bigl).  
\end{align*}
Modulo $t^2$, we have 
$$\sum_{i=1}^{p-1}s_i\bigl(\varphi(x),t\varphi_1(x)\bigl)=\frac{1}{p-1}[\underset{p-2}{\underbrace{\varphi(x),\cdots,[\varphi(x)}},[t\varphi_1(x),\varphi(x)]\cdots].$$
Thus, by collecting the coefficients of $t$, we obtain the identity

\begin{equation}
\varphi(\w_1(x))+\varphi_1(x^{[p]})=\epsilon_1(\varphi(x))+\ad_{\varphi(x)}^{p-1}\circ\varphi_1(x).
\end{equation}
Therefore, $\bigl((\mu_1,\w_1),(\nu_1,\epsilon_1),\varphi_1\bigl)\in\fZ^2_{*}(\varphi,\varphi)$.\\~\\

\noindent\textbf{Obstructions.} Let $\bigl(L,[ \cdot , \cdot ]_L,(\cdot )^{[p]_L}\bigl)$ and $\bigl(L,[ \cdot , \cdot ]_M,(\cdot )^{[p]_M}\bigl)$ be restricted Lie algebras. Let $(\mu_t^n,\w_t^n)$ (resp. $(\nu_t^n,\epsilon_t^n)$) be order $n$ restricted deformation of $L$ (resp. $M$). Let $\varphi:L\rightarrow M$ be a restricted morphism and let $\varphi_t^n$ be order $n$ restricted deformation of $\varphi$, that is, $\displaystyle\varphi_t^n=\sum_{i\geq0}^nt^i\varphi_i$. For all $x,y\in L$, we define
\begin{equation}
   \Obs^{(1)}_n(\varphi)(x,y):=\sum_{i=1}^n\Bigl(\nu_i(\varphi_{n+1-i}(x),\varphi(y))-\varphi_i(\mu_{n+1-i}(x,y))+\sum_{j=0}^i\nu_j(\varphi_{i-j}(x),\varphi_{n+1-i}(y))\Bigl). 
\end{equation}

Suppose that the deformations are infinitesimal, that is, $n=1$. We will investigate the obstructions to extend the deformations to order $2$. Let $\varphi^2_t=(\cdot)^{[p]}+t\varphi_1+t^2\varphi_2$, with $\varphi_2:L\rightarrow M$. Let $x\in L$. We define
    \begin{equation}
        \Obs_2^{(2)}(\varphi)(x):=x^{p-1}\cdot\varphi_1(x)-\varphi\circ\w_1(x)-\sum_{i+j=p-2}x^i\cdot[\varphi_1(x),x^j\cdot\varphi_1(x)].
    \end{equation}
        The following computations are made modulo $t^3$.
\begin{align}
\bigl(\varphi(x)+t\varphi_1(x)+t^2\varphi_2(x)\bigl)^{[p]}&=\varphi(x)^{[p]}+t^p(\varphi_1(x)+t\varphi_2(x))^{[p]}+\sum_{i=1}^{p-1}\bigl(\varphi(x),t\varphi_1(x)+t^2\varphi_2(x)\bigl)\nonumber\\
&=\varphi(x)^{[p]}+\frac{1}{p-1}[\underset{p-2}{\underbrace{\varphi(x),[\cdots,[\varphi(x)}},[t\varphi_1(x)+t^2\varphi_2(x),\varphi(x)]\cdots]\nonumber\\
&~~+\frac{1}{p-2}\sum_{i=1}^{p-2}[\varphi(x),\cdots,[\underset{\text{position }i}{\underbrace{t\varphi_1(x)}},\cdots,[t\varphi_1(x),\varphi(x)]\cdots]\nonumber\\
&=\varphi(x)^{[p]}+t\ad^{p-1}_{\varphi(x)}\circ\varphi_1(x)+t^2\ad^{p-1}_{\varphi(x)}\circ\varphi_2(x)\nonumber\\&~~+t^2\frac{1}{p-2}\sum_{i=1}^{p-2}[\varphi(x),\cdots,[\underset{\text{position }i}{\underbrace{\varphi_1(x)}},\cdots,[\varphi_1(x),\varphi(x)]\cdots]\nonumber\\
&=\varphi(x)^{[p]}-t\ad^{p-1}_{\varphi(x)}\circ\varphi_1(x)
\nonumber\\&\label{obslem1}~~+t^2\Bigl(\ad^{p-1}_{\varphi(x)}\circ\varphi_2(x)-\frac{1}{p-2}\sum_{i+j=p-2}x^i\cdot[\varphi_1(x),x^j\cdot\varphi_1(x)]\Bigl).
\end{align}
Moreover, we have that
\begin{equation}\label{obslem2}
    \epsilon_1\bigl(\varphi(x)+t\varphi_1(x)\bigl)=\epsilon_1\circ\varphi(x)-t\ad^{p-1}_{\varphi(x)}\circ\varphi_1(x).
\end{equation}
Suppose that $\varphi_t^2$ is a restricted morphism. Then, from $\varphi_t^2\circ\w_t^2=\epsilon_t^2\circ\varphi_t^2$ and using Eqs.\eqref{obslem1} and \eqref{obslem2}, we obtain that
\begin{align}\label{eqobs}
\varphi\circ\w_2(x)+\varphi_1\circ\w_1(x)+\varphi_2(x^{[p]})&=\ad_{\varphi(x)}^{p-1}\circ(\varphi_2-\varphi_1)(x)+\epsilon_2\circ\varphi(x)\\\nonumber
        &~~-\frac{1}{p-2}\sum_{i+j=p-2}x^i\cdot[\varphi_1(x),x^j\cdot\varphi_1(x)],~\forall x\in L.
\end{align}
Therefore, we have the following result.
\begin{prop}
    Let $\bigl(L,[ \cdot , \cdot ]_L,(\cdot )^{[p]_L}\bigl)$ and $\bigl(L,[ \cdot , \cdot ]_M,(\cdot )^{[p]_M}\bigl)$ be restricted Lie algebras. Let $(\mu_t=[\cdot,\cdot]_L+t\mu_1,\w_t=(\cdot)^{[p]_L}+t\w_1)$ $($resp. $(\nu_t=[\cdot,\cdot]_M+t\nu_1,\epsilon_t=(\cdot)^{[p]_M}+t\epsilon_1)$$)$ be an infinitesimal restricted deformation of $L$ (resp. $M$). Let $\varphi:L\rightarrow M$ be a restricted morphism and let $\varphi_t=\varphi+t\varphi_1$ be an infinitesimal restricted deformation of $\varphi$. Suppose that there exists maps $\mu_2,\nu_2,\w_2,\epsilon_2,\varphi_2$ such that 
    \begin{enumerate}
        \item[$(i)$] $(\mu_t+t^2\mu_2,\w_t+t^2\w_2)$ is a restricted deformation of $L$;
        \item[$(ii)$] $(\nu_t+t^2\nu_2,\epsilon_t+t^2\epsilon_2)$ is a restricted deformation of $M$;
        \item[$(iii)$] $\varphi_t+t^2\varphi_2$ is a restricted deformation of $\varphi$.
    \end{enumerate}
    Then, $\Obs_2^{(1)}(\varphi)=-\alpha_{\mu_1,\nu_1}(\varphi_2)$ and $\Obs_2^{(2)}(\varphi)=-\beta_{\w_1,\epsilon_1}(\varphi_2).$
   
\end{prop}

\begin{proof}
    The proof follows from Eq.\eqref{eqobs} and definitions of $\alpha$ and $\beta$.
\end{proof}

\noindent\textbf{Equivalence.} Let $\bigl(L,[ \cdot , \cdot ]_L,(\cdot )^{[p]_L}\bigl)$ and $\bigl(M,[ \cdot , \cdot ]_M,(\cdot )^{[p]_M}\bigl)$ be restricted Lie algebras and let $\varphi:L\rightarrow M$ be a restricted morphism. Let $(\mu_t,\w_t)$ and $(\tilde{\mu_t},\tilde{\w_t})$ (resp. $(\nu_t,\epsilon_t)$ and $(\tilde{\nu_t},\tilde{\epsilon_t})$) be restricted deformations of $L$ (resp. $M$). Let $\phi_t$ (resp. $\psi_t$) be an equivalence of deformations between $(\mu_t,\w_t)$ and $(\tilde{\mu_t},\tilde{\w_t})$ (resp. $(\nu_t,\epsilon_t)$ and $(\tilde{\nu_t},\tilde{\epsilon_t})$). Finally, let $$\varphi_t:\bigl(L[[t]],\mu_t,\w_t\bigl)\rightarrow\bigl(M[[t]],\nu_t,\epsilon_t\bigl) \text{ and } \widetilde{\varphi_t}:\bigl(L[[t]],\tilde{\mu_t},\tilde{\w_t}\bigl)\rightarrow\bigl(M[[t]],\tilde{\nu_t},\tilde{\epsilon_t}\bigl)$$ be deformations of the morphism $\varphi$.
In summary, we have the following diagram.
\begin{equation}
 \begin{tikzcd}
{\bigl(L[[t]],\mu_t,\w_t\bigl)} \arrow[dd, "\phi_t"'] \arrow[rr, "\varphi_t"]    &  & {\bigl(M[[t]],\nu_t,\epsilon_t\bigl)} \arrow[dd, "\psi_t"] \\
                                                                                 &  &                                            \\
{\bigl(L[[t]],\tilde{\mu_t},\tilde{\w_t}\bigl)} \arrow[rr, "\widetilde{\varphi_t}"'] &  & {\bigl(M[[t]],\tilde{\nu_t},\tilde{\epsilon_t}\bigl)}     
\end{tikzcd}
\end{equation}

For all $x\in L$, we have the relations
\begin{align*}
\phi_t\circ\mu_t(x)=&~ \Tilde{\mu_t}\bigl(\phi_t(x),\phi_t(x)\bigl),~\phi_t\circ\w_t(x)=\Tilde{\w_t}\circ\phi_t(x);\\
\psi_t\circ\nu_t(x)=& ~\Tilde{\nu_t}\bigl(\psi_t(x),\psi_t(x)\bigl),~\psi_t\circ\epsilon_t(x)=\Tilde{\epsilon_t}\circ\psi_t(x).
\end{align*}
\begin{defi}\label{defiequimorph}
The deformations $\varphi_t$ and $\widetilde{\varphi_t}$ are called \textit{equivalent} if $\widetilde{\varphi_t}\circ\phi_t=\psi_t\circ\varphi_t$. A deformation $\varphi_t$ is called \textit{trivial} if it is equivalent to the deformation $\widetilde{\varphi_t}\equiv\varphi$.
\end{defi}
Expanding $\widetilde{\varphi_t}\circ\phi_t=\psi_t\circ\varphi_t$, we obtain in particular
\begin{equation}
  \varphi_1-\widetilde{\varphi_1}=\varphi\circ\phi_1-\psi_1\circ\varphi=\alpha_{\phi_1,\psi_1}(0).
\end{equation}
Therefore, $\varphi_1-\widetilde{\varphi_1}$ is a coboundary.

\begin{prop}\label{propequimorph}
    Let $\varphi_t$ a deformation of $\varphi:L\rightarrow M$. Let $n\geq 1$ and suppose that $\varphi_t$ is given by
    $\varphi_t=\varphi+t^n\varphi_n,$ where $\varphi_n$ is a coboundary, that is, there exists $f:L\rightarrow L$ and $g:M\rightarrow M$ such that $\varphi_n=\varphi\circ f+g\circ\varphi$. Then the deformation $\varphi_t$ is equivalent to a deformation $\widetilde{\varphi_t}$ such that $\widetilde{\varphi_i}=0,~\forall i\leq n$. Therefore, any deformation $\varphi_t$ such that all $\varphi_i$ are coboundaries is trivial.
\end{prop}
\begin{proof}
    Consider $\phi_t=\id+t^nf$ and $\psi_t=\id+t^ng$. We build a deformation $$\widetilde{\varphi_t}:=\psi_t\circ\varphi_t\circ\phi_t^{-1}=\varphi+\sum_{i\geq1}t^i\widetilde{\varphi_i}.$$ Let $x\in L$. We have
    \begin{align}
\widetilde{\varphi_t}\circ\phi_t(x)&=\varphi(x)+t^n\varphi\circ f(x)+\sum_{i\geq 1}t^i\widetilde{\varphi_i}(x)+\sum_{i\geq 1}t^{i+n}\widetilde{\varphi_i}\circ f(x);\\
        \psi_t\circ\varphi_t(x)&=\varphi(x)+t^n\varphi_n(x)+t^n g\circ\varphi(x)+t^{2n}\varphi_n(x).
    \end{align}
Since $\widetilde{\varphi_t}\circ\phi_t=\psi_t\circ\varphi_t$, we obtain that $\widetilde{\varphi_i}=0,~\forall i\leq n$.

\end{proof}

	\subsection{Restricted Nijenhuis operators}
	We briefly consider the restricted version of Nijenhuis operators. Let $(L,[\cdot,\cdot],(\cdot)^{[p]})$ be a restricted Lie algebra.
	\begin{defi}
		A linear map $N: L\rightarrow L$ is called a \textit{restricted Nijenhuis operator} on $L$ if
		\begin{align}
			N\bigl([N(x),y]+[x,N(y)]-N([x,y])\bigl)&=[N(x),N(y)]\\
			N\bigl(N(x^{[p]})-\ad_x^{p-1}\circ N(x)\bigl)&=N(x)^{[p]},~~~~~~~~~~~\forall x,y\in L.
		\end{align}		
	\end{defi}	
We define two maps on $L$ by
	\begin{align}
		[x,y]_N&:=[N(x),y]+[x,N(y)]-N([x,y])\\
		x^{[p]_N}&:=-N(x^{[p]})+\ad_x^{p-1}\circ N(x).
	\end{align}
	
	\begin{prop}
		The pair $\left([\cdot,\cdot]_N,(\cdot)^{[p]}\right)$ is a restricted $2$-cocycle. Moreover, the restricted formal deformation given by
		\begin{equation}
			[x,y]_t=[x,y]+t[x,y]_N,~~x^{[p]_t}=x^{[p]}+t x^{[p]_N}
		\end{equation}
		is trivial.
	\end{prop}
	
	\begin{proof}
		By definition, we have
			$[x,y]_N=d^1_{\text{CE}}N(x,y)$ and $
			x^{[p]_N}=\ind^1N(x).$ Therefore, $\bigl([\cdot,\cdot]_N,(\cdot)^{[p]}\bigl)$ is a restricted $2$-coboundary.
	\end{proof}

 \black
	
	\section{Restricted Lie algebras in characteristic $2$}\label{R2}
	
	From now, $\F$ denotes a field of characteristic $p=2$. This section aims at studying the specific case of restricted Lie algebras over a field of characteristic $p=2$. We describe a cohomology for restricted Lie algebras which turns to be different from that of Evans and Fuchs (\cite{EF08}). 
	
	\subsection{Definition}
	In characteristic $p=2$, the Definition \ref{restdefi} of a restricted Lie algebra reduces to the following.
	
	\begin{defi}\label{restlie2}
		A \textit{restricted Lie algebra} in characteristic $p=2$ is a Lie algebra $(L,[\cdot,\cdot])$ endowed with a map  $(\cdot)^{[2]}:L\longrightarrow L$ (called $2$-map) such that
		\begin{enumerate}
			\item $(\lambda x)^{[2]}=\lambda^{2}x^{[2]}$, for all $x\in L$ and for all $\lambda\in \F$;
			\item $\left[x,y^{[2]} \right]=[[x,y],y]$, for all $x,y\in L;$
			\item $(x+y)^{[2]}=x^{[2]}+y^{[2]}+[x,y]$, for all $x,y\in L$.
		\end{enumerate}
		
	\end{defi}

	\begin{prop}
		Let $L$ be a restricted Lie algebra in characteristic $p=2$. 
		
		\begin{itemize}
			\item Let $x_1,\cdots,x_n\in L$. Then we have the formula $$\left(\sum_{i=1}^{n}x_i\right)^{[2]}=\sum_{i=1}^{n}x_i^{[2]}+\sum_{1\leq i<j\leq n}[x_i,x_j]. $$
			\item Suppose that the adjoint representation on $L$ is faithful. Then, Conditions 1. and 3. of Definition \ref{restlie2} follow from Condition 2.
		\end{itemize}

	\end{prop}
	
	\begin{proof} The first point follows from a straightforward computation. Let $x,y,z\in L$ and
		suppose that the adjoint representation $\ad:x\mapsto \ad_x=[x,\cdot]$ is faithful. Let $\lambda\in\F$. We have
		$$\ad_{(\lambda x)^{[2]}}(y)=[(\lambda x)^{[2]},y]=[\lambda x,[\lambda x,y]]=\lambda^2[x,[x,y]]=\lambda^2[x^{[2]},y]=\lambda^2\ad_{x^{[2]}}(y).$$ Therefore, we have $(\lambda x)^{[2]}=\lambda^{2}x^{[2]}$. Then,
		\begin{align*}
			\ad_{(x+y)^{[2]}}(z)&=[x+y,[x+y,z]]\\
			&=[x,[x,z]]+[y,[y,z]]+[x,[y,z]]+[y,[x,z]]\\
			&=[x^{[2]},z]+[y^{[2]},z]+[[x,y],z]\\
			&=\ad_{x^{[2]}}(z)+\ad_{y^{[2]}}(z)+\ad_{[x,y]}(z).
		\end{align*}
		It follows that $(x+y)^{[2]}=x^{[2]}+y^{[2]}+[x,y],~\forall x,y\in L$.
	\end{proof}
	
	\subsection{Semi-direct product in characteristic $2$}

	Let $\left(L,[\cdot,\cdot],(\cdot)^{[2]}\right)$ and $\left(\g,[\cdot,\cdot]_{\g}, (\cdot)^{[2]_{\g}}\right)$ be two restricted Lie algebras.
	
	\begin{prop}\label{2semidirect}
	Let $\pi:L\longrightarrow \Der(\g)$ be a restricted map such that \begin{equation}\label{restrictedder}\pi(x)\bigl(g^{[2]_{\g}}\bigl)=[\pi(x)(g),g]_{\g},~\forall x\in L,~g\in\g.\end{equation}
Then, the vector space $L\oplus\g$ is a restricted Lie algebra with the bracket
		\begin{equation}
			\bigl[(x,g),(y,h)\bigl]_{\pi}:=\bigl([x,y],\pi(x)(h)+\pi(y)(g)+[g,h]_{\g}\bigl)
		\end{equation} and the $2$-map
		\begin{equation}
			(x,g)^{[2]_{\pi}}:=\bigl(x^{[2]},\pi(x)(g)+g^{[2]_{\g}}\bigl).
		\end{equation}
	\end{prop}
	
	\begin{proof}
		The proof that $[\cdot,\cdot]_{\pi}$ is a Lie bracket follows from \cite[Section 2.2]{CGL18}.
		Let us show that the map $(\cdot)^{[2]_{\pi}}$ is a $2$-map with respect to the bracket $[\cdot,\cdot]_{\pi}$.
		Let $\lambda\in\F$, $x,y\in L$ and $g,h\in\g$. We have 
		$$\bigl(\lambda(x,g)\bigl)^{[2]_{\pi}}=\Bigl((\lambda x)^{[2]},\pi(\lambda x)(\lambda g)+(\lambda g)^{[2]_{\g}}\Bigl)=\Bigl(\lambda^2 ( x)^{[2]},\lambda^2\pi(x)(g)+\lambda^2(g)^{[2]_{\g}}\Bigl)=\lambda^2(x,g)^{[2]_{\pi}}.$$
		Then, we have
		\begin{align*}
			&\bigl((x,g)+(y,h)\bigl)^{[2]_{\pi}}-(x,g)^{[2]_{\pi}}-(y,h)^{[2]_{\pi}}\\
			=&\bigl((x+y),(g+h)\bigl)^{[2]_{\pi}}-(x,g)^{[2]_{\pi}}-(y,h)^{[2]_{\pi}}\\
			=&\bigl((x+y)^{[2]},\pi(x+y)(g+h)+(g+h)^{[2]_{\g}}\bigl)-\bigl(x^{[2]},\pi(x)                  (g)+(g)^{[2]_{\g}}\bigl)-\bigl(y^{[2]},\pi(y)(h)+(h)^{[2]_{\g}}\bigl)\\
			=&\bigl([x,y], \pi(x)(g)+\pi(x)(h)+\pi(y)(g)+\pi(y)(h)+g^{[2]_{\g}}+ h^{[2]_{\g}}+[g,h]_{\g} -   \pi(x)(g)-g^{[2]_{\g}}-\pi(y)(h)-h^{[2]_{\g}} \bigl)\\
			=&\bigl([x,y], \pi(x)(h)+\pi(y)(g)+[g,h]_{\g}\bigl)\\
			=&\bigl[(x,g),(y,h)\bigl]_{\pi}.
		\end{align*}\normalsize{}
		Finally, we obtain
		\begin{align*}
			&\bigl[(x,g),[(x,g),(y,h)]_{\pi}\bigl]_{\pi}-\bigl[(x,g)^{[2]_{\pi}},(y,h)\bigl]\\
			=& \bigl[(x,g),\left([x,y],\pi(x)(h)+\pi(y)(g)+[g,h]_{\g} \right)\bigl]_{\pi}-\bigl[ \bigl(x^{[2]},\pi(x)(g)+g^{[2]_{\g}}\bigl),(y,h) \bigl]_{\pi}\\
			=&\bigl([x,[x,y]],\pi(x)\bigl(\pi(x)(h)+\pi(y)(g)+[g,h]_{\g}\bigl)+\pi([x,y])(g)+\bigl[g,\pi(x)(h)+\pi(y)(g)+[g,h]_{\g}\bigl]_{\g}\bigl)\\
			&+\Bigl(\bigl[x^{[2]},y\bigl], \pi(x^{[2]})\bigl(h)+\pi(y)\bigl(\pi(x)(g)+g^{[2]_{\g}}\bigl)+\bigl[\pi(x)(g)+g^{[2]_{\g}},h  \bigl]_{\g}   \Bigl).
		\end{align*}
		The first component gives $[x^{[2]},y]-[x,[x,y]]=0$. Moreover, we have:
		\begin{align*}
			\bullet~& \pi([x,y])(g)+\pi(x)\circ\pi(y)(g)+\pi(y)\circ\pi(x)(g)=0 \text{ since $\pi$ is a Lie morphism;}\\
			\bullet~& \pi(x)^2(h)-\pi\bigl(x^{[2]}\bigl)=0 \text{ since $\pi$ is restricted;}\\
			\bullet~&  \pi(x)\bigl([g,h]_{\g}\bigl)+\bigl[g,\pi(x)(h) \bigl]_{\g}+ \bigl[\pi(x)(g),h\bigl]_{\g} =0 \text{ since $\pi(x)$ is a derivation of $\g$ for all $x\in\g$; }\\   
			\bullet~& \pi(x)\bigl(g^{[2]_{\g}}\bigl)-[\pi(x)(g),g]_{\g}=0 \text{ using Eq.\eqref{restrictedder}; }\\
			\bullet~& [g^{[2]_{\g}},h]_{\g}-[g,[g,h]_{\g}]_{\g}=0.
		\end{align*}
		
		Therefore, we obtain $\bigl[(x,g),[(x,g),(y,h)]_{\pi}\bigl]_{\pi}-\bigl[(x,g)^{[2]_{\pi}},(y,h)\bigl]=0$. Therefore, $(\cdot)^{[2]_{\pi}}$ is a $2$-map on $L\oplus\g$ with respect to the bracket $[\cdot,\cdot]_{\pi}$.  
	\end{proof}
	
In the case where Proposition \ref{2semidirect} holds, the restricted Lie algebra $\bigl(L\oplus\g,[\cdot,\cdot]_{\pi},(\cdot)^{[2]_{\pi}}\bigl)$ is called the \textit{semi-direct product of $L$ and $\g$}.

	%
	%
		%

	\subsection{$2$-mappings versus formal power series}
	
	Let $\left( L,[\cdot,\cdot],(\cdot)^{[2]}\right)$ be a restricted Lie algebra. The formal space $L[[t]]:=\Bigl\{\displaystyle\sum_it^ ix_i,~x_i\in L\Bigl\}$ is a Lie algebra with the bracket
	\begin{equation}\label{extbracket}
		\Bigl[\sum_{i\geq 0}t^ix_i,\sum_{j\geq 0}t^jy_j \Bigl]=\sum_{i,j}t^{i+j}[x_i,y_j],~~ \forall x_i,y_j\in L. 
	\end{equation}
	Now, we aim to extend the map $(\cdot)^{[2]}$ on $L[[t]]$ in such a way that $L[[t]]$ is endowed with a restricted Lie algebra structure with respect to the extended bracket defined in \eqref{extbracket}.
	Let $x_i\in L,~n\in \N$ and $\lambda\in\F$. We have
	\begin{equation}\label{scalarformula}
		\Bigl(\sum_{i=0}^{n}\lambda^ix_i\Bigl)^{[2]}=\sum_{i=0}^{n}\lambda^{2i}x_i^{[2]}+\sum_{0\leq i<j\leq n}\lambda^{i+j}[x_i,x_j].
	\end{equation}
	\begin{prop}
		Let $L$ be a Lie algebra. Then $L[[t]]$ is a restricted Lie algebra with the extended bracket (\ref{extbracket}) and the $2$-mapping $(\cdot)^{[2]_t}$ given by 
		
		\begin{equation}\label{extend2map}\Bigl(\sum_{i\geq 0}t^ix_i \Bigl)^{[2]_{t}}:=\sum_{i\geq 0}t^{2i}x_i^{[2]}+\sum_{i,j}t^{i+j}[x_i,x_j].    \end{equation}
	\end{prop}
	\begin{proof}
		
		We check the three conditions of the Definition \ref{restlie2}. Let $\lambda\in \F$ and $x_i, y_j\in L$. First, we have   
            \begin{align*} \Bigl(\lambda\sum_it^ix_i\Bigl)^{[2]_{t}}&=\Bigl(\sum_it^i(\lambda x_i)\Bigl)^{[2]_{t}}\\
				&=\sum_it^{2i}(\lambda x_i)^{[2]}+\sum_{i<j}t^{i+j}[\lambda x_i,\lambda x_j]\\
				&=\lambda^2\sum_it^{2i}x_i^{[2]}+\lambda^2\sum_{i<j}t^{i+j}[x_i,x_j]\\
				&=\lambda^2\Bigl(\sum_it^i x_i\Bigl)^{[2]_{t}}.
			\end{align*}
			 Then, we have \begin{align*}
				\Bigl[\sum_it^ix_i,\Bigl(\sum_jt^j y_j\Bigl)^{[2]_t} \Bigl]&=\Bigl[\sum_it^ix_i,\sum_jt^{2j}y_j^{[2]}\Bigl]+\Bigl[\sum_it^ix_i,\sum_{j<k}t^{j+k}[y_j,y_k]\Bigl]\\
				&=\sum_{i,j}t^{i+2j}\left[x_i,y_j^{[2]}\right]+\sum_{\underset{j<k}{i,j,k}}t^{i+j+k}[x_i,[y_j,y_k]]\\
				&=\sum_{i,j}t^{i+2j}\left[x_i,y_j^{[2]}\right]+\sum_{\underset{j<k}{i,j,k}}t^{i+j+k}[y_j,[y_k,x_i]]+\sum_{\underset{j<k}{i,j,k}}t^{i+j+k}[y_k,[x_i,y_j]]\\
				&=\sum_{i,j}t^{i+2j}\left[x_i,y_j^{[2]}\right]+\sum_{\underset{j<k}{i,j,k}}t^{i+j+k}[[x_i,y_k],y_j]+\sum_{\underset{j<k}{i,j,k}}t^{i+j+k}[[x_i,y_j],y_k]\\
				&=\sum_{i,j}t^{i+2j}\left[x_i,y_j^{[2]}\right]+\sum_{\underset{j>k}{i,j,k}}t^{i+j+k}[[x_i,y_j],y_k]+\sum_{\underset{j<k}{i,j,k}}t^{i+j+k}[[x_i,y_j],y_k]\\
				&=\sum_{i,j}t^{i+j+j}[[x_i,y_j],y_j]+\sum_{\underset{j\neq k}{i,j,k}}t^{i+j+k}[[x_i,y_j],y_k]\\
				&=\sum_{i,j,k}t^{i+j+k}[[x_i,y_j],y_k]\\
				&=\bigl[\Bigl[\sum_it^ix_i,\sum_jt^jy_j\Bigl],\sum_jt^jy_j\bigl].
			\end{align*}
			The following computation will be useful to prove the last remaining identity.	
			\begin{align}
				\nonumber\sum_{i\neq j}t^{i+j}[x_i,y_j]&=\sum_{i<j}t^{i+j}[x_i,y_j]+\sum_{j<i}t^{i+j}[x_i,y_j]\\
				\nonumber&=\sum_{i<j}t^{i+j}[x_i,y_j]+\sum_{i<j}t^{i+j}[x_j,y_i]\\
				&=\sum_{i<j}t^{i+j}[x_i,y_j]+\sum_{i<j}t^{i+j}[y_i,x_j].\label{eqqq}
			\end{align}
			Now we can prove the third condition:
			\begin{align*}
				\Bigl(\sum_it^ix_i+\sum_j t^j y_j\Bigl)^{[2]_{t}}&=\Bigl(\sum_it^i(x_i+y_i)\Bigl)^{[2]_{t}}\\
				&=\sum_it^{2i}(x_i+y_i)^{[2]}+\sum_{i<j}t^{i+j}[x_i+y_i,x_j+y_j]\\
				&=\sum_it^{2i}x_i^{[2]}+\sum_it^{2i}y_i^{[2]}+\sum_it^{2i}[x_i,y_i]\\
				&~~+\sum_{i<j}t^{i+j}[x_i,x_j]+\sum_{i<j}t^{i+j}[x_i,y_j]\\
				&~~+\sum_{i<j}t^{i+j}[x_j,y_i]+\sum_{i<j}t^{i+j}[x_j,y_j]\\
				&=\sum_it^{2i}x_i^{[2]}+\sum_it^{2i}y_i^{[2]}+\sum_it^{2i}[x_i,y_i]\\
				&~~+\sum_{i<j}t^{i+j}[x_i,x_j]+\sum_{i<j}t^{i+j}[x_j,y_j]\\
				&~~+\sum_{i\neq j}t^{i+j}[x_i,y_j]~~\text{~~~~(using Eq.(\ref{eqqq}))}\\
				&=\Bigl(\sum_it^ix_i\Bigl)^{[2]_{t}}+\Bigl(\sum_it^iy_i\Bigl)^{[2]_{t}}+\sum_{i,j}t^{i+j}[x_i,y_j]\\
				&=\Bigl(\sum_it^ix_i\Bigl)^{[2]_{t}}+\Bigl(\sum_it^iy_i\Bigl)^{[2]_{t}}+\Bigl[\sum_it^ix_i,\sum_j t^j y_j\Bigl].
			\end{align*}	
	\end{proof}
	
	\begin{rmq} By expanding the formula (\ref{extend2map}) and by arranging the terms by monomials of the same degree, we obtain
		\begin{equation}
			\Bigl(\sum_{n\geq 0}t^nx_n \Bigl)^{[2]_{t}}=\sum_{n\geq 0}t^n\Bigl( \epsilon(n)x^{[2]}_{\lfloor\frac{n}{2}\rfloor}+\sum_{\underset{i+j=n}{i<j}}[x_i,x_j] \Bigl), 
		\end{equation}
		where $\lfloor\cdot\rfloor$ denotes the floor function, $\epsilon(n)=0$ if $n$ is odd and $\epsilon(n)=1$ if $n$ is even.
	\end{rmq}

	
	\subsection{Cohomology of restricted Lie algebras in characteristic $2$}\label{seccoh2}

	Let $\bigl(L,[\cdot,\cdot],(\cdot)^{[2]}\bigl)$ be a restricted Lie algebra and let $M$ be a restricted $L$-module. We start by setting $C_{*_2}^0(L,M):=C_{\text{CE}}^0 (L,M)$ and $C_{*_2}^1(L,M):=C_{\text{CE}}^1 (L,M)$.
	
	\begin{defi}
		Let $n\geq2$, $\varphi\in C_{\text{CE}}^n (L,M)$, $\omega:L\times \wedge^{n-2}L\rightarrow M$, $\lambda\in\F$ and $x,z_2,\cdots,z_{n-1}\in L$. The pair $(\varphi,\omega)$ is a $n$-cochain of the restricted cohomology if
		\begin{align}
			\omega(\lambda x, z_2,\cdots,z_{n-1})&=\lambda^2\omega(x,z_2,\cdots,z_{n-1}),\\
			\omega(x,z_2,\cdots,\lambda z_i+z_i',\cdots,z_{n-1})&=\lambda\omega(x,z_2,\cdots,z_i,\cdots,z_{n-1})+\omega(x,z_2,\cdots,z_i',\cdots,z_{n-1}),\\
			\omega(x+y,z_2,\cdots,z_{n-1})&=\omega(x,z_2,\cdots,z_{n-1})+\omega(y,z_2,\cdots,z_{n-1})+\varphi(x,y,z_2,\cdots,z_{n-1}).
		\end{align}
		We denote by $C_{*_2}^n(L,M)$ the space of $n$-cochains of $L$ with values in $M$.	
	\end{defi}
	The coboundary maps
	$d^n_{*_2}:C^n_{*_2}(L,M)\rightarrow C^{n+1}_{*_2}(L,M)$ for $n\geq 2$, are given by $d^n_{*_2}(\varphi,\omega)=\bigl(d^n_{\text{CE}}(\varphi),\delta^n(\omega)\bigl),$ where
	\begin{align*}	\delta^n\omega(x,z_2,\cdots,z_n)&:=x\cdot\varphi(x,z_2,\cdots,z_n)+\sum_{i=2}^{n}z_i\cdot\omega(x,z_2,\cdots,\hat{z_i},\cdots,z_n)\\	&+\varphi(x^{[2]},z_2,\cdots,z_n)+\sum_{i=2}^{n}\varphi\left([x,z_i],x,z_2,\cdots,\hat{z_i},\cdots,z_n \right)\\
		&+\sum_{1\leq i<j\leq n}\omega\left(x,[z_i,z_j],z_2,\cdots,\hat{z_i},\cdots,\hat{z_j},\cdots,z_n  \right).
	\end{align*}
	
	\begin{lem}
		Let $n\geq2$ and $(\varphi,\omega)\in C^n_{*_2}(L,M)$. Then $\bigl(d^n_{\text{CE}}(\varphi),\delta^n(\omega)\bigl)\in C^{n+1}_{*_2}(L,M)$. 
	\end{lem}
	
	\begin{proof}
		Let $x,y,z_2,\cdots,z_{n+1}\in L$. We show that \begin{equation}\label{omega*ppty}\delta^n\omega(x+y,z_2,\cdots,z_{n-1})=\delta^n\omega(x,z_2,\cdots,z_{n-1})+\delta^n\omega(y,z_2,\cdots,z_{n-1})+d_{\text{CE}}^n\varphi(x,y,z_2,\cdots,z_{n-1}).	 \end{equation} 
  We have
\begin{align*}
\delta^n\omega(x+y,z_2,\cdots,z_n)&=x\cdot\varphi(x+y,z_2,\cdots,z_n)
			+\sum_{i=2}^{n}z_i\cdot\omega(x+y,z_2,\cdots,\hat{z_i},\cdots,z_n)\\
			&~~+\varphi((x+y)^{[2]},z_2,\cdots,z_n)
			+\sum_{i=2}^{n}\varphi\left([x+y,z_i],x+y,z_2,\cdots,\hat{z_i},\cdots,z_n \right)\\
			&~~+\sum_{1\leq i<j\leq n}\omega\left(x+y,[z_i,z_j],z_2,\cdots,\hat{z_i},\cdots,\hat{z_j},\cdots,z_n  \right).\\ 
			&=~x\cdot\varphi(x,z_2,\cdots,z_n)+x\cdot\varphi(y,z_2,\cdots,z_n)\\
            &~~+y\cdot\varphi(x,z_2,\cdots,z_n)+y\cdot\varphi(y,z_2,\cdots,z_n)\\
			&~~+\sum_{i=2}^{n}z_i\cdot\omega(x,z_2,\cdots,\hat{z_i},\cdots,z_n)+\sum_{i=2}^{n}z_i\cdot\omega(y,z_2,\cdots,\hat{z_i},\cdots,z_n)\\
			&~~+\sum_{i=1}^{n}\varphi(x,y,z_2,\cdots,\hat{z_i},\cdots,z_n)\\
			&~~+\varphi(x^{[2]},z_2,\cdots,z_n)+\varphi(y^{[2]},z_2,\cdots,z_n)+\varphi([x,y],z_2,\cdots,z_n)\\
			&~~+\sum_{i=2}^{n}\varphi\left([x,z_i],x,z_2,\cdots,\hat{z_i},\cdots,z_n \right)+\sum_{i=2}^{n}\varphi\left([x,z_i],y,z_2,\cdots,\hat{z_i},\cdots,z_n \right)\\
			&~~+\sum_{i=2}^{n}\varphi\left([y,z_i],y,z_2,\cdots,\hat{z_i},\cdots,z_n \right)+\sum_{i=2}^{n}\varphi\left([y,z_i],x,z_2,\cdots,\hat{z_i},\cdots,z_n \right)\\
			&~~+\sum_{2\leq i<j\leq n}\omega\left(x,[z_i,z_j],z_2,\cdots,\hat{z_i},\cdots,\hat{z_j},\cdots,z_n \right)\\
			&~~+\sum_{2\leq i<j\leq n}\omega\left(y,[z_i,z_j],z_2,\cdots,\hat{z_i},\cdots,\hat{z_j},\cdots,z_n \right)\\
			&~~+\sum_{2\leq i<j\leq n}\varphi(x,z_i,z_j,z_2,\cdots,\hat{z_i},\cdots,\hat{z_j},\cdots,z_n).
		\end{align*}
		We can now identify the desired terms in the above expression:
		\begin{align*}	
			\delta^n\omega(x,z_2,\cdots,z_n)&=x\cdot\varphi(x,z_2,\cdots,z_n)+\sum_{i=2}^{n}z_i\cdot\omega(x,z_2,\cdots,\hat{z_i},\cdots,z_n)+\varphi(x^{[2]},z_2,\cdots,z_n)\\
			&+\sum_{i=2}^{n}\varphi\left([x,z_i],x,z_2,\cdots,\hat{z_i},\cdots,z_n \right)\\
			&+\sum_{2\leq i<j\leq n}\omega\left(x,[z_i,z_j],z_2,\cdots,\hat{z_i},\cdots,\hat{z_j},\cdots,z_n \right);\\
			\delta^n\omega(y,z_2,\cdots,z_n)&=y\cdot\varphi(y,z_2,\cdots,z_n)+\sum_{i=2}^{n}z_i\cdot\omega(y,z_2,\cdots,\hat{z_i},\cdots,z_n)+\varphi(y^{[2]},z_2,\cdots,z_n)\\
			&+\sum_{i=2}^{n}\varphi\left([y,z_i],y,z_2,\cdots,\hat{z_i},\cdots,z_n \right)\\
			&+\sum_{2\leq i<j\leq n}\omega\left(y,[z_i,z_j],z_2,\cdots,\hat{z_i},\cdots,\hat{z_j},\cdots,z_n \right);\\
			d^n_{\text{CE}}\varphi(x,y,z_2,\cdots,z_n)
			&=x\cdot\varphi(y,z_2,\cdots,z_n)+y\cdot\varphi(x,z_2,\cdots,z_n)+\sum_{i=1}^{n}\varphi(x,y,z_2,\cdots,\hat{z_i},\cdots,z_n)\\
			&+\sum_{i=1}^{n}\varphi([x,y],z_2,\cdots,z_n)+\sum_{i=2}^{n}\varphi\left([x,z_i],y,z_2,\cdots,\hat{z_i},\cdots,z_n \right)\\
			&+\sum_{i=2}^{n}\varphi\left([y,z_i],x,z_2,\cdots,\hat{z_i},\cdots,z_n \right)\\
			&+\sum_{2\leq i<j\leq n}\varphi(x,z_i,z_j,z_2,\cdots,\hat{z_i},\cdots,\hat{z_j},\cdots,z_n).\\
		\end{align*}
		Equation (\ref{omega*ppty}) is then satisfied.
	\end{proof}
	
	\begin{lem}
		Let $n\geq 2$ and $(\varphi,\omega)\in C^n_{*_2}(L,M)$. We have $\delta^{n+1}\circ\delta^n=0$.
	\end{lem}
	
	\begin{proof}
		Let $x,z_2,\cdots, z_{n+1}\in L$. We have
		\begin{align*}
			&\delta^{n+1}\circ\delta^n\w(x,z_2,\cdots,z_{n+1})\\
			&~~=x\cdot d_{\text{CE}}^n\varphi(x,z_2,\cdots,z_{n+1})
			+\sum_{i=2}^{n+1}z_i\cdot\delta^n\w(x,z_2,\cdots,\hat{z_i},\cdots,z_{n+1})\\
			&~~+d_{\text{CE}}^n\varphi(x^{[2]},z_2,\cdots,z_{n+1}
			+\sum_{i=2}^{n+1}d_{\text{CE}}^n\varphi([x,z_i],x,z_2,\cdots,\hat{z_i},\cdots,z_{n+1})\\
			&~~+\sum_{2\leq i<j\leq n}\delta^n\w(x,[z_i,z_j],z_2,\cdots,\hat{z_i},\cdots,\hat{z_j},\cdots,z_{n+1})\\
			&=\sum_{i=2}^{n+1}z_i\cdot\left( x\cdot\varphi(x,z_2,\cdots,\hat{z_i},\cdots,z_{n+1})\right)
			+\sum_{i=2}^{n+1} z_i\cdot\sum_{\underset{j\neq i}{j=2}}z_j\cdot\w(x,z_2,\cdots,\hat{z_i},\cdots\hat{z_j},\cdots,z_{n+1})\\
			&~~+\sum_{i=2}^{n+1}z_i\cdot\varphi(x^{[2]},z_2,\cdots,\hat{z_i},\cdots,z_{n+1})
			+\sum_{i=2}^{n+1}z_i\cdot\sum_{\underset{j\neq i}{j=2}}^{n+1}\varphi([x,z_j],x,z_2,\cdots,\hat{z_i},\cdots,\hat{z_j},\cdots,z_{n+1})\\
			&~~+\sum_{i=2}^{n+1}z_i\cdot\sum_{\underset{j,k\neq i}{2\leq j<k\leq n+1}}\w(x,[z_j,z_k],z_2,\cdots,\hat{z_i},\cdots,\hat{z_j},\cdots,\hat{z_k}\cdots,z_{n+1})\\
			&~~+\sum_{i=2}^{n+1}z_i\cdot\varphi(x^{[2]},z_2,\cdots,z_{n+1})+x^{[2]}\cdot\varphi(z_2,\cdots,z_{n+1})
			\\&~~+\sum_{2\leq i<j\leq n+1}\varphi([z_i,z_j],x^{[2]},\cdots,\hat{z_i},\cdots,\hat{z_j},\cdots,z_{n+1})\\
			&~~+\sum_{j=2}^{n+1}\varphi([x^{[2]},z_j],z_2,\cdots,\hat{z_j},...,z_{n+1})
			+x\cdot\sum_{i=2}^{n+1}z_i\cdot\varphi(x,z_2,\cdots,\hat{z_i},\cdots,z_{n+1})+x\cdot(x\cdot\varphi(z_2,\cdots,z_{n+1}))\\
			&~~+x\cdot\sum_{2\leq i<j\leq n+1}\varphi([z_i,z_j],x,\cdots,\hat{z_i},\cdots,\hat{z_j},\cdots,z_{n+1})
			+x\cdot\sum_{j=2}^{n+1}\varphi([x,z_j],z_2,\cdots,\hat{z_j},\cdots,z_{n+1})\\
			&~~+\sum_{i=2}^{n+1}\sum_{\underset{j\neq i}{j=2}}^{n+1}z_j\cdot\varphi([x,z_i],x,z_2,\cdots,\hat{z_i},\cdots,\hat{z_j},\cdots,z_{n+1})
			+\sum_{i=2}^{n+1}x\cdot\varphi([x,z_i],z_2,\cdots,\hat{z_i},\cdots,z_{n+1})\\
			&~~+\sum_{i=2}^{n+1}[x,z_i]\cdot\varphi(x,z_2,\cdots,\hat{z_i},\cdots,z_{n+1})
			\\
			&~~+\sum_{i=2}^{n+1}\sum_{\underset{j,k\neq i}{2\leq j<k\leq n+1}}\varphi([z_j,z_k],[x,z_i],x,z_2,\cdots,\hat{z_i},\cdots,\hat{z_j},\cdots,\hat{z_k}\cdots,z_{n+1}))\\
			&~~+\sum_{i=2}^{n+1}\sum_{\underset{j\neq i}{j=2}}^{n+1}\varphi([x,z_j],[x,z_i],\cdots,\hat{z_i},\cdots,\hat{z_j},\cdots,z_{n+1})
			\\&~~+\sum_{i=2}^{n+1}\sum_{\underset{j\neq i}{j=2}}^{n+1}\varphi([[x,z_i],z_j],x,z_2,\cdots,\hat{z_i},\cdots,\hat{z_j},\cdots,z_{n+1})\\
			&~~+\sum_{i=2}^{n+1}\varphi([[x,z_i],x],z_2,\cdots,\hat{z_i},\cdots,z_{n+1})
			+\sum_{2\leq i<j\leq n+1}x\cdot\varphi(x,[z_i,z_j],\cdots,\hat{z_i},\cdots,\hat{z_j},\cdots,z_{n+1})\\
			&~~+\sum_{2\leq i<j\leq n+1}\sum_{\underset{j,k\neq i}{k=2}}^{n+1}z_k\cdot\w(x,[z_i,z_j]),z_2,\cdots,\hat{z_i},\cdots,\hat{z_j},\cdots,\hat{z_k}\cdots,z_{n+1})\\
			&~~+\sum_{2\leq i<j\leq n+1}[z_i,z_j]\cdot\w(x,z_2\cdots,\hat{z_i},\cdots,\hat{z_j},\cdots,z_{n+1})
			\\
			&~~+\sum_{2\leq i<j\leq n+1}\varphi(x^{[2]},[z_i,z_j],z_2,\cdots\hat{z_i},\cdots,\hat{z_j},\cdots,z_{n+1})\\
			&~~+\sum_{2\leq i<j\leq n+1}\varphi([x,[z_i,z_j],x,z_2,\cdots\hat{z_i},\cdots,\hat{z_j},\cdots,z_{n+1})\\
			&~~+\sum_{2\leq i<j\leq n+1}\sum_{\underset{j,k\neq i}{k=2}}^{n+1}\varphi([x,z_k],x,[z_i,z_j],z_2,\cdots,\hat{z_i},\cdots,\hat{z_j},\cdots,\hat{z_k}\cdots,z_{n+1})\\
			&~~+\sum_{2\leq i<j\leq n+1}~~\sum_{\underset{k,l\neq i,j}{2\leq k<l\leq n+1}}\w(x,[z_k,z_l],z_2,\cdots,\hat{z_i},\cdots,\hat{z_j},\cdots,\hat{z_k}\cdots,\hat{z_l},\cdots,z_{n+1})\\
			&~~+\sum_{2\leq i<j\leq n+1}^{n+1}\sum_{\underset{k\neq i,j}{k=2}}^{n+1}\w(x,[[z_i,z_j],z_k],z_2,\cdots,\hat{z_i},\cdots,\hat{z_j},\cdots,\hat{z_k}\cdots,z_{n+1})=0. 
		\end{align*}
	\end{proof}
	Thus, we have obtained a cochain complex $\left(C_{*_2}^n(L,M),d^n_{*_2} \right)_{n\geq 2}. $ For $n=0,1$, we define $d^0_{*_2}=d^0_{\text{CE}}$ and 
	\begin{align*}
		d^1_{*_2}:C_{*_2}^1(L,M)&\longrightarrow C_{*_2}^2(L,M)\\
		\varphi&\longmapsto\bigl(d_{\text{CE}}^1\varphi,\omega\bigl),~\omega(x):=\varphi\bigl( x^{[2]} \bigl)+x\cdot\varphi(x),~\forall x\in L.
	\end{align*} 
	
	\begin{lem}
		The map	$d^1_{*_2}$ is well-defined. We have $d^1_{*_2}\circ d^0_{*_2}=0 \text{ and } d^2_{*_2}\circ d^1_{*_2}=(0,0).$
	\end{lem}
	
	
	Our cochain complex is now complete.
	
	\begin{thm}\label{cohomology2}
		Let $\bigl(L,[\cdot,\cdot],(\cdot)^{[2]}\bigl)$ be a restricted Lie algebra and $M$ be a restricted $L$-module. The complex $\left(C_{*_2}^n(L,M),d^n_{*_2} \right)_{n\geq 0}$ is a cochain complex. The $n^{th}$ restricted cohomology group of the Lie algebra $L$ in characteristic 2 is defined by	
		$$ H_{*_2}^n(L,M):=Z_{*_2}^n(L,M)/B^n_{*_2}(L,M),$$ 
		with $Z_{*_2}^n(L,M)=\Ker(d^n_{*_2})$ the restricted $n$-cocycles and $B_{*_2}^n(L,M)=\im(d^{n-1}_{*_2})$ the restricted $n$-coboundaries.\footnote{Apparently, an earlier instance of those formulae can be found in May's papers \cite[for example]{MJP66} and are known by experts, but we did not found them written in the above explicit way.}
	\end{thm}

	\begin{rmq} 
		$H^0_{*_2}(L,M)=H^0_{\text{CE}}(L,M)$.
	\end{rmq}
	\begin{rmq}
		This cohomology has no analogue in characteristic different from $2$. Very similar cohomology formulas have been considered in \cite{BM22}, in the slightly different context of (Hom-) Lie superalgebras of characteristic $2$.
	\end{rmq}
	
	\subsection{Computations in small degrees}\label{compu2}
	
	Hereafter, we give some applications of the cohomology of restricted Lie algebras in characteristic 2 defined above.
	
	\subsubsection{First cohomology group and restricted derivations.}
	
	We recall that a restricted derivation $D$ of a restricted Lie algebra $\bigl(L,[\cdot,\cdot],(\cdot)^{[2]}\bigl)$ in characteristic $2$ is a linear map $D:L\rightarrow L$ that satisfies $D([x,y])=[D(x),y]+[x,D(y)]$ and $D(x^{[2]})=[x,D(x)]$ for all $x,y\in L$. Let $\varphi$ be a restricted $1$-cocycle, that reads, for $x,y\in L$:
	\begin{center}
		$\begin{cases}
			\varphi([x,y])=[\varphi(x),y]+[x,\varphi(y)];\text{ and }\\ 
			~~\varphi(x^{[2]})=[x,\varphi(x)].
		\end{cases}$
	\end{center}
It is clear that any $1$-cocyle $\varphi$ with values in $L$ is a restricted derivation. We have
	$$B^1_{*_2}(L,L)=B^1_{\text{CE}}(L,L)=\im(d^0_{\text{CE}})=\left\lbrace \ad_x,~x\in L \right\rbrace.$$
	Every derivation of the form $\ad_x$ is restricted. Those derivations are called \textit{inner derivations}. Therefore, we have 
	$$ H^1_{*_2}(L,L)=Z^1_{*_2}(L,L)/B^1_{*_2}(L,L)=\left\lbrace\text{restricted derivations} \right\rbrace /\left\lbrace\text{inner derivations} \right\rbrace.$$
	This is a well-known result in the case where $p>2$ (see \cite{ET00,EF08}).
	
	\subsubsection{Second cohomology group with scalar coefficients and central extensions.} Let $\bigl(L,[\cdot,\cdot]_L,(\cdot)^{[2]_L}\bigl)$ be a restricted Lie algebra and let $\g:=L\oplus \F c$, where $c$ is a parameter. Here, $\F$ is viewed as a trivial $L$-module. A restricted scalar $2$-cocycle is a pair $(\varphi,\w)\in C^2_{*_2}(L,\F)$ that satisfies for all $x,y,z\in L$,
	\begin{equation}
		\varphi(x,[y,z]_L)+\varphi(y,[z,x]_L)+\varphi(z,[x,y]_L)=0
	\end{equation}
	and
	\begin{equation}\label{scalar2cocycle}
		\varphi\bigl(x,y^{[2]_L}\bigl)=\varphi([x,y]_L,y). 
	\end{equation}
	Let $x,y\in L$ and $u,v\in \F$. We define a bracket on $\g$ by
	\begin{equation}\label{bracketextensionC}
		[x+uc,y+vc]_{\g}:=[x,y]_L+\varphi(x,y)c
	\end{equation}
	and a map $(\cdot)^{[2]_{\g}}:\g\rightarrow\g$ by
	\begin{equation}\label{2mapextensionC}
		(x+uc)^{[2]_{\g}}:=x^{[2]_L}+\w(x)c.
	\end{equation}
	
	\begin{prop} Let $\g=L\oplus \F c$ equipped with the bracket \eqref{bracketextensionC} and the 2-map \eqref{2mapextensionC}. Then,
		$\left(\g,[\cdot,\cdot]_{\g},(\cdot)^{[2]_{\g}}\right)$ is a restricted Lie algebra if and only if $(\varphi,\w)$ is a restricted $2$-cocycle.
	\end{prop}
	
	\begin{proof}
		It is well-known in the ordinary case that $\left(\g,[\cdot,\cdot]_{\g}\right)$ is a Lie algebra if and only if $\varphi$ is a Chevalley-Eilenberg $2$-cocycle. It remains to show that $(\cdot)^{[2]_{\g}}$ is a $2$-mapping on $\g$ if and only if Eq.(\ref{scalar2cocycle}) is satisfied. Let $x,y\in L$ and $u,v\in \F$. We have
		\begin{align*}
			\left((x+u)+(y+v)\right)^{[2]_{\g}}&=(x+y)^{[2]_L}+\w(x+y)c\\
			&=x^{[2]_L}+[x,y]_L+\w(x)c+\w(y)c+\varphi(x,y)c\\
			&=(x+uc)^{[2]_{\g}}+(y+vc)^{[2]_{\g}}+\bigl[(x+uc),(y+vc)\bigl]_{\g}.
		\end{align*} Moreover, we have
		\begin{align*}
			\bigl[(x+uc),(y+vc)^{[2]_{\g}}]_{\g}&=[(x+uc),y^{[2]_L}+\w(y)c\bigl]_{\g}\\
			&=\bigl[[x,y^{[2]_L}\bigl]_L+\varphi\bigl(x,y^{[2]_L}\bigl)c\\
			&=\bigl[[x,y]_L,y\bigl]_L+\varphi([x,y]_L,y)c\\
			&=\bigl[[x,y]_L+\varphi(x,y)c,y+vc\bigl]_{\g}\\
			&=\bigl[[x+uc,y+vc]_{\g},y+vc\bigl]_{\g}.
		\end{align*}
		Finally, it is clear that $(\lambda(x+uc))^{[2]_{\g}}=\lambda^2(x+uc)^{[2]_{\g}}.$ Therefore, we conclude that $(\cdot)^{[2]_{\g}}$ is a $2$-mapping on $\g$ if and only if Eq.(\ref{scalar2cocycle}) is satisfied.
		
	\end{proof}
	
	

	\subsection{Restricted formal deformations}\label{defo2section}
	
	The aim of this section is to consider restricted formal deformations of restricted Lie algebras in characteristic $p=2$. In particular, we show that the deformation theory is controlled by the cohomology introduced in Section \ref{seccoh2}.
	
	\begin{defi}\label{defidefo2}
		Let $\bigl(L,[\cdot,\cdot], (\cdot)^{[2]}\bigl)$ be a restricted Lie algebra. A restricted formal deformation of $\bigl(L,[\cdot,\cdot], (\cdot)^{[2]}\bigl)$ is given by two maps
		\begin{equation*}
            \begin{array}{lllll}
                m_t:&L\times L\longrightarrow L[[t]]& \text{and }\quad & \w_t:&L\longrightarrow L[[t]] \\[2mm]
                &(x,y)\longmapsto \displaystyle\sum_{i\geq 0}t^i m_i(x,y)&\quad  &&x\longmapsto \displaystyle\sum_{j\geq 0}t^j\omega_j(x), \\[2mm]
             \end{array}
    \end{equation*}
		where $m_0=m$, $\w_0=\w$, $m_i:L\times L\rightarrow L$ antisymmetric and $\w_i:L\mapsto L$ satisfying $\w(\lambda x)=\lambda^2\w(x),~\forall \lambda\in\F,x\in L$.
		
		Moreover, $m_t$ and $\w_t$ must satisfy the following equations, for all $x,y,z\in L$,
		
		\begin{equation}\label{eqdef1}
			m_t(x,m_t(y,z))+m_t(y,m_t(z,x))+m_t(z,m_t(x,y))=0;
		\end{equation}
		\begin{equation}\label{eqdef2}
			m_t(x,\w_t(y))=m_t(m_t(x,y),y);
		\end{equation}
		\begin{equation}\label{eqdef3}
			\w_t(x+y)=\w_t(x)+\w_t(y)+m_t(x,y);
		\end{equation}

	\end{defi}
	
	\begin{rmq}
		\hspace{0.3cm}
		\begin{enumerate}
			\item The map $m_t$ extends to $L[[t]]$ by $\F[[t]]$-linearity.
			
			\item The map $\w_t$ extends to $L[[t]]$ using Eqs.(\ref{extend2map}) and (\ref{eqdef3}).
		\end{enumerate}
	\end{rmq}

	\begin{lem}
		Let $\left(m_t,\w_t \right)$ be a restricted deformation of $\bigl(L,[\cdot,\cdot], (\cdot)^{[2]}\bigl)$. Then $\left(m_k,\w_k \right)\in C_{*_2}^2(L,L)~\forall k\geq 0$.
	\end{lem}
	
	\begin{proof}
		Let $x,y\in L$. Expanding Eq.(\ref{eqdef3}) yields
		$$\sum_{i\geq 0}t^i\omega_i(x+y)=\sum_{i\geq 0}t^i\omega_i(x)+\sum_{i\geq 0}t^i\omega_i(y)+\sum_{i\geq 0}t^im_i(x).$$ Then, for every $k\geq 0$, we have
		$$\omega_k(x+y)=\omega_k(x)+\omega_k(y)+m_k(x,y),$$ which is the desired identity. Moreover, for $\lambda\in \F$, we have	
		$$ \w_t(\lambda x)=\sum_{i\geq 0}t^i\omega_i(\lambda x)=\lambda^2\sum_{i\geq 0}t^i\omega_i(x), \text{ so } \omega_i(\lambda x)=\lambda^2\omega_i(x),~\forall i\geq0.  $$
	\end{proof}
	
	\begin{prop}
		Let $\left(m_t,\w_t \right)$ be a restricted deformation of $\bigl(L,[\cdot,\cdot], (\cdot)^{[2]}\bigl)$. Then $(m_1,\w_1)$ is a $2$-cocycle of the restricted cohomology, that reads
		$$ d_{\ce}^2m_1=0 \text{ and } \delta^2\w_1=0.    $$
	\end{prop}
	
	\begin{proof}
		The ordinary theory ensures that $d^2_{\text{CE}}m_1=0$. It remains to check that $\delta^2\w_1=0$. By expanding Eq.(\ref{eqdef2}), we obtain
		
		\begin{equation}
			\sum_{i,j}t^{i+j}m_i(x,\w_j(y))=\sum_{i,j}t^{i+j}m_i(m_j(x,y),y).
		\end{equation}
		Collecting the coefficients of $t$ yields
		$$m_1(x,y^{[2]})+[x,\w_1(y)]=[m_1(x,y),y]+m_1([x,y],y),  $$
		which is equivalent to $\delta^2\w_1=0.$
	\end{proof}
	
	\subsubsection{Equivalence of restricted formal deformations}
	
	Let $\phi_t:L[[t]]\longrightarrow L[[t]]$ be a formal automorphism defined on $L$ by 
	$$ \phi_t(x)=\sum_{i\geq 0}t^i\phi_i(x),\text{ with }\phi_i:L\rightarrow L\text{ and }\phi_0=id,$$ then extended by $\F[[t]]$-linearity.
	
	\begin{defi}
		Two formal deformations $\left(m_t,\w_t \right)$ and $\left(m'_t,\w'_t \right)$ of $\bigl(L,[\cdot,\cdot], (\cdot)^{[2]}\bigl)$ are \textit{equivalent} if there exists a formal automorphism $\phi_t$ such that
		\begin{align}\label{eqeq1}
			m_t'\left(\phi_t(x),\phi_t(y) \right)&=\phi_t\left(m_t(x,y) \right)\quad\text{  and  }  \\
		\label{eqeq2}
			\phi_t\left(\omega_t(x) \right)&=\w'_t\left(\phi_t(x)\right),~\quad~~\forall x,y\in L. 
		\end{align}

	\end{defi}
	
	\begin{lem}\label{my}
		With the above data, we have
		\begin{align}
			m_1(x,y)+m'_1(x,y)&=\phi_1([x,y])+[x,\phi_1(y)]+[\phi_1(x),y];\\	\omega_1(x)+\w'_1(x)&=[\phi_1(x),x]+\phi_1\left(\w(x)\right),~\forall x,y\in L.
		\end{align}
	\end{lem}
	
	\begin{proof}
		
		Equation \eqref{eqeq1} is equivalent to
		
		$$m_t'\Bigl(\sum_{i\geq0}t^i\phi_i(x),\sum_{j\geq0} t^j\phi_j(y)\Bigl)-\sum_{k\geq0}t^k\phi_k(m_t(x,y))=0.$$
		Therefore, we have
		$$\sum_{i,j,k\geq0}t^{i+j+k}m_k'\left(\phi_i(x),\phi_j(y)\right)=\sum_{i,j\geq0}t^{i+j}\phi_j(m_i(x,y)).$$
		By collecting the coefficients of $t$, we obtain
		$$  [\phi_1(x),y]+[x,\phi(y)]+m_1(x,y)=\phi_1([x,y])+m'_1(x,y),  $$ which is equivalent to the first desired equation. The following computations are made mod $t^2$. Eq.\eqref{eqeq2} gives $$\phi_t\bigl(\sum_{i\geq0}t^i\w_i(x)\bigl)=\w'_t\left(x+t\phi_1(x)\right),$$ which implies
		$$\sum_{i\geq0}t^i\phi_t(\w_i(x))=\w'_t(x)+\w'_t(t\phi_1(x))+m'_t(x,t\phi_1(x)).$$ Therefore,
		$$\sum_{i,j\geq0}t^{i+j}\phi_j(w_i(x))=x^{[2]}+t\bigl(\w'_1(x)+[x,\phi_1(x)]\bigl).$$	
		By collecting the coefficients of $t$, we obtain	
		$$\w_1(x)+\phi_1(x^{[2]})=\w'_1(x)+[x,\phi_1(x)],$$ which is equivalent to the second desired equation.
	\end{proof}
	
	\begin{prop}
		If $\left(m_t,\omega_t \right) $ and $\left( m'_t,\w'_t \right) $ are two equivalent deformations of $\bigl(L,[\cdot,\cdot], (\cdot)^{[2]}\bigl)$, then their infinitesimals $\left(m_1,\omega_1 \right) $ and $\left( m'_1,\w'_1 \right) $ are cohomologous.
	\end{prop}
	
	\begin{proof}
		Notice that
		$$\phi_1([x,y])+[x,\phi_1(y)]+[\phi_1(x),y]=d_{\text{CE}}^1\phi_1(x,y)  $$ and 
		$$[\phi_1(x),x]+\phi_1(x^{[2]})=\delta^1\phi_1$$ in Lemma \ref{my}.
	\end{proof}
	
	\begin{defi}
		A formal deformation  $\left(m_t,\omega_t \right)$  of $\bigl(L,[\cdot,\cdot], (\cdot)^{[2]}\bigl)$ is called \textbf{trivial} if there is a formal automorphism $\phi_t$ satisfying
		\begin{align}
			\phi_t\left(m_t(x,y) \right)&=[\phi_t(x),\phi_t(y)];\\
			\phi_t\left(\omega_t(x) \right)&=\left(\phi_t(x)\right)^{[2]},~\forall x,y\in L.    
		\end{align}
	\end{defi}
	
	\begin{prop}
		Suppose that $(m_1,\omega_1)\in B^2_{*_2}(L,L)$. Then the infinitesimal deformation of $\bigl(L,[\cdot,\cdot], (\cdot)^{[2]}\bigl)$ given by $m_t=[\cdot,\cdot]+tm_1$ and $\omega_t=(\cdot)^{[2]}+t\omega_1$ is trivial.
	\end{prop}
	
	\begin{proof}
		Suppose that $m_1\in B^2_{*_2}(L,L)$, that is, there exists $\varphi: L\rightarrow L$ such that $m_1=d_{\text{CE}}^2\varphi$ and $\w_1=\delta^1\varphi$. We consider a formal automorphism
		$\phi_t=\id+t\varphi.$
		Since $m_1$ is a Chevalley-Eilenberg $2$-coboundary, we have
		\begin{equation}
			m_1(x,y)=\varphi([x,y])+[x,\varphi(y)]+[\varphi(x),y].
		\end{equation}
		Thus, we can write
		$$[x,y]+t\Bigl(\varphi([x,y])+m_1(x,y) \Bigl)=[x,y]+t\Bigl([x,\varphi(y)]+[\varphi(x),y]\Bigl),$$
		which is equivalent (mod $t^2$) to
		$$\phi_t\Bigl([x,y]+tm_1(x,y)\Bigl)=[\phi_t(x),\phi_t(y)].$$
		Finally, we obtain (mod $t^2$)
		\begin{equation}\label{prooftriv1}
			\phi_t(m_t(x,y))=[\phi_t(x),\phi_t(y)].
		\end{equation}
		Then, using the identity $\w_1=\delta^1\varphi$, we have
		\begin{equation}
			\w_1(x)+\varphi(x^{[2]})=[x,\varphi(x)].
		\end{equation}
		Thus, we can write
		$$ x^{[2]}+t\Bigl(\w_1(x)+\phi( x^{[2]})\Bigl)=x^{[2]}+t[x,\varphi(x)],$$
		which implies (mod $t^2$) that
		$$\phi_t\Bigl(\w(x)+t\w_1(x)\Bigl)=(x+t\varphi(x))^{[2]}.$$
		Finally, we obtain (mod $t^2$)
		\begin{equation}\label{prooftriv2}
			\phi_t\left(\w_t(x)\right)=\phi_t(x)^{[2]}.
		\end{equation}	
		Eqs.(\ref{prooftriv1}) and (\ref{prooftriv2}) together implies that the deformation is trivial.
		
	\end{proof}
	
	
	\subsubsection{Obstructions}
	
	Let $\bigl(L,[\cdot,\cdot], (\cdot)^{[2]}\bigl)$ be a restricted Lie algebra.
	A restricted deformation $(m^n_t,\w^n_t)$ of $\bigl(L,[\cdot,\cdot], (\cdot)^{[2]}\bigl)$ of order $n\in \N$ is given by truncated formal power series
	
	$$m_t^n=\sum_{k=0}^{n}t^km_k~\text{ and }~\w_t^n=\sum_{k=0}^{n}t^k\omega_k.$$
	
	\begin{defi}
		For all $x,y,z\in L$, we define the following quantities:
		\begin{align*}
			\obs_{n+1}^{(1)}(x,y,z)&=\sum_{i=1}^{n}\bigl(m_i(x,m_{n+1-i}(y,z))+m_i(y,m_{n+1-i}(z,x))+m_i(z,m_{n+1-i}(x,y)) \bigl);\\ 
			\obs_{n+1}^{(2)}(x,y)&=\sum_{i=1}^{n}\bigl( m_i(y,\omega_{n+1-i}(x))+m_i(m_{n+1-i}(y,x),x) \bigl).
		\end{align*}
	\end{defi}
	
	\begin{lem}
		$\bigl( \obs_{n+1}^{(1)},\obs_{n+1}^{(2)}\bigl)~\in C^3_{*_2}(L,L)$. 
	\end{lem}
	
	\begin{proof}
		Let $x_1,x_2,y\in L.$
		\begin{align*}
			\obs_{n+1}^{(2)}(x_1+x_2,y)&=\sum_{i=1}^{n}\left( m_i(y,\omega_{n+1-i}(x_1+x_2))+m_i\left( m_{n+1-i}(y,x_1+x_2),x_1+x_2 \right)\right) \\
			&=\sum_{i=1}^{n}\left( m_i\left(y,\omega_{n+1-i}(x_1)\right) + m_i\left(y,\omega_{n+1-i}(x_2) \right)+m_i\left(y,m_{n+1-i}(x_1,x_2)\right) \right)  \\
			&~~+\sum_{i=1}^{n}\left( m_i\left(m_{n+1-i}(y,x_1),x_1 \right) +m_i\left(m_{n+1-i}(y,x_1),x_2 \right)\right) \\
			&~~+\sum_{i=1}^{n}\left( m_i\left(m_{n+1-i}(y,x_2),x_1 \right)+m_i\left(m_{n+1-i}(y,x_2),x_2 \right)\right)  \\
			&=\sum_{i=1}^{n}\left( m_i\left(y,\omega_{n+1-i}(x_1)\right) + m_i\left(m_{n+1-i}(y,x_1),x_1 \right)\right)\\
			&~~+\sum_{i=1}^{n}\left( m_i\left(y,\omega_{n+1-i}(x_2)\right) + m_i\left(m_{n+1-i}(y,x_2),x_2 \right)\right)\\
			&~~+\sum_{i=1}^{n}\left(m_i\left(m_i(y,m_{n+1-i}(x_1,x_2) \right)+m_1\left(m_{n+1-i}(y,x_1),x_2 \right)+m_1\left(m_{n+1-i}(y,x_2),x_1 \right)\right)\\
			&=\obs_{n+1}^{(2)}(x_1,y)+\obs_{n+1}^{(2)}(x_2,y)+\obs_{n+1}^{(1)}(x_1,x_2,y).
		\end{align*}	
	\end{proof}

	\begin{prop}
		
		Let $\left( m^n_t,\w^n_t \right) $ be a $n$-order deformation of $\bigl(L,[\cdot,\cdot], (\cdot)^{[2]}\bigl)$. Let $\left(m_{n+1},~\w_{n+1} \right)\in C^2_{*_2}(L,L)$. Suppose that $\left( m^n_t+t^{n+1}m_{n+1},~  \w^n_t+t^{n+1}\w_{n+1} \right) $ is a $(n+1)$-order deformation of $L$. Then
		$$	\left( \obs^{(1)}_{n+1}, \obs^{(2)}_{n+1}\right)=d_{*_2}^2\left(m_{n+1},~\w_{n+1} \right).        $$
		
	\end{prop}
	
	\begin{proof}
		Suppose that $m^n_t+t^{n+1}m_{n+1}  $
		satisfies the Jacobi identity. Then, we obtain for $x,y,z\in L$ that
		\begin{equation}
			\sum_{i=0}^{n+1}\left(m_i(x,m_{n+1-i}(y,z)) +m_i(y,m_{n+1-i}(z,x)) +m_i(z,m_{n+1-i}(x,y))\right)=0,
		\end{equation}
		which can be rewritten	
		\begin{equation}
			\sum_{i=1}^{n}\left(m_i(x,m_{n+1-i}(y,z)) +m_i(y,m_{n+1-i}(z,x)) +m_i(z,m_{n+1-i}(x,y))\right)=d_{\text{CE}}^2m_{n+1}(x,y,z).
		\end{equation} 
		Conversely, suppose that $\obs^{(1)}_{n+1}=d_{\text{CE}}^2m_{n+1}$. Then $m^n_t+t^{n+1}m_{n+1}$ satisfies the Jacobi identity. Now suppose that $\w^n_t+t^{n+1}\w_{n+1}$ is a $2$-map with respect to $m^n_t+t^{n+1}m_{n+1}$. The following equation is then satisfied:
		
		\begin{equation}\label{obseq}
			m_t^{n+1}\left(x,\w_t^{n+1}(y) \right)=m_t^{n+1}\left( m_t^{n+1}(x,y),y\right),  
		\end{equation} 
		where we have denoted $m_t^{n+1}:=m^n_t+t^{n+1}m_{n+1}$ and $\w^{n+1}_t=\w^n_t+t^{n+1}\w_{n+1}$.
		By expanding Eq.(\ref{obseq}), we obtain
		
		\begin{equation}
			\sum_{k=0}^{n+1}t^k\sum_{i=0}^{k}m_i(x,\omega_{k-i}(y))=\sum_{k=0}^{n+1}t^k\sum_{i=0}^{k} m_i(m_{k-i}(x,y),y).
		\end{equation}	 
		By collecting the coefficients of $t^{n+1}$, we obtain
		
		$$
		\sum_{i=0}^{n+1}m_i(x,\omega_{n+1-i}(y))=\sum_{i=0}^{n+1} m_i(m_{n+1-i}(x,y),y),
		$$
		which can be rewritten
		\begin{align*}
			[x,\omega_{n+1}(y)]+m_{n+1}(x,y^{[2]})&+[m_{n+1}(x,y),y]+m_{n+1}([x,y],y)\\
			&=\sum_{i=1}^{n}\left(m_i(x,\omega_{n+1-i}(y))+ m_i(m_{n+1-i}(x,y),y)\right). 
		\end{align*}
		We conclude that
		
		$$\obs_{n+1}^{(2)}(x,y)=\sum_{i=1}^{n}\Bigl(m_i(x,\omega_{n+1-i}(y))+ m_i(m_{n+1-i}(x,y),y)\Bigl)=\delta_{*_2}^2\omega_{n+1}(x,y).$$
		
	\end{proof}

\subsection{Deformation of restricted morphisms in characteristic $2$}\label{2morpfdefo2}
	
	In this section, we investigate restricted deformations of restricted morphisms in characteristic $p=2$. We provide a cohomology controlling those deformations which is specific of the characteristic $2$ case.

		\subsubsection{Deformation cohomology of restricted morphisms in characteristic $2$}

	Let $(L,[ \cdot , \cdot ]_L,(\cdot )^{[2]_L})$ and $(M,[ \cdot , \cdot ]_M,(\cdot )^{[2]_M})$ be two restricted Lie algebras. Let $\varphi:L\rightarrow M$ be a restricted morphism and let $n\geq1$. We define 
	$$\fC^n_{*_2}(\varphi,\varphi):=C^n_{*_2}(L,L)\times  C^n_{*_2}(M,M)\times C^{n-1}_{*_2}(L,M),$$
	and $\fC^0_{*_2}(\varphi,\varphi):=0$. For $n\geq 3$, the differential maps are given by
	\begin{align}
		\fd^n_{*_2}:\fC^n_{*_2}(\varphi,\varphi)&\rightarrow\fC^{n+1}_{*_2}(\varphi,\varphi)\nonumber\\\label{diffmorph2}
		\begin{pmatrix}(\mu,\w)\\(\nu,\epsilon)\\(\theta,\rho)\end{pmatrix}&\mapsto \begin{pmatrix}\bigl(d^n_{\ce}\mu,\delta^n\w\bigl)\\\bigl(d^n_{\ce}\nu,\delta^n\epsilon\bigl)\\\bigl(\alpha_{\mu,\nu}(\theta),\beta_{\w,\epsilon}(\rho)\bigl)\end{pmatrix},
	\end{align}
	where $\alpha_{\mu,\nu}(\theta):=\varphi\circ\mu+\nu\circ\varphi^{\otimes n}+d^{n-1}_{\ce}\theta$ and $\beta_{\w,\epsilon}(\rho):=\varphi\circ\w+\epsilon\circ \varphi^{\otimes(n-1)}+\delta^{n-1}\rho$.
     Moreover, we have
    \begin{align}
        \fd^1:\fC^1_{*_2}(\varphi,\varphi)&\rightarrow\fC^{2}_{*_2}(\varphi,\varphi)\\\nonumber
        (\mu,\nu,m)&\mapsto \Bigl((d^1_{\ce}\mu,\delta^1\mu),(d^1_{\ce}\nu,\delta^1\nu),\alpha_{\mu,\nu}(m)  \Bigl)
    \end{align} and
    \begin{align}
        \fd^2:\fC^2_{*_2}(\varphi,\varphi)&\rightarrow\fC^{3}_{*_2}(\varphi,\varphi)\\\nonumber
        \Bigl((\mu,\w),(\nu,\epsilon),\theta\Bigl)&\mapsto \Bigl((d^2_{\ce}\mu,\delta^2\w),(d^2_{\ce}\nu,\delta^2\epsilon),\bigl(\alpha_{\mu,\nu}(\theta),\beta_{\mu,\nu}(\theta)\bigl)  \Bigl).
    \end{align}

    We denote by $\mathfrak{Z}^n_{*_2}(\varphi,\varphi):=\Ker(\fd^n_{*_2})$ and $\mathfrak{B}^n_{*_2}(\varphi,\varphi):=\im(\fd^{n-1}_{*_2})$, the $n$-cocycles and $n$-coboundaries, respectively.\\

\begin{thm}\label{cohomorph2}
    For all $n\in\N$, the maps $\fd^n_{*_2}$ are well defined and satisfy $\fd^{n+1}_{*_2}\circ\fd^n_{*_2}=0$.
\end{thm}

\begin{proof}
   Let $\Bigl( (\mu,\w),(\nu,\epsilon),(\theta,\rho)\Bigl)\in\fC^n_{*_2}(\varphi,\varphi)$ and let $x,y, z_2,\cdots, z_{n-2}\in L$. We denote by $z:=(z_2,\cdots, z_{n-2})$, $\hat{z_i}:=(z_2,\cdots,\hat{z_i},\cdots z_{n-2})$ and $\hat{\hat{z}}_{i,j}:=(z_2,\cdots,\hat{z_i},\cdots,\hat{z_j},\cdots z_{n-2})$.
   \begin{align*}
        \beta_{\w,\epsilon}(\rho)(x+y,z)&=\varphi\circ\w(x,z)+\varphi\circ\w(y,z)+\underline{\varphi\circ\mu(x,y,z)}\\
        &~~+\epsilon\circ\varphi(x,z)+\epsilon\circ\varphi(y,z)+\underline{\nu\circ\varphi(x,y,z)}\\
        &~~+x\cdot\theta(x,z)+\underline{x\cdot\theta(y,z)+y\cdot\theta(x,z)}+y\cdot\theta(y,z)\\
        &~~+\sum_{i=2}^{n-2}z_i\cdot\bigl(\rho(x,\hat{z_i})+\rho(y,\hat{z_i})+\underline{\theta(x,y,\hat{z_i})}\bigl)\\
        &~~+\theta(x^{[2]},z)+\theta(y^{[2]},z)+\underline{\theta([x,y],z)}\\
        &~~+\sum_{i=2}^{n-2}\theta\bigl([x,z_i],x,\hat{z_i}\bigl)+\underline{\sum_{i=2}^{n-2}\theta\bigl([y,z_i],x,\hat{z_i}\bigl)}\\
        &~~+\underline{\sum_{i=2}^{n-2}\theta\bigl([y,z_i],x,\hat{z_i}\bigl)}+\sum_{i=2}^{n-2}\theta\bigl([y,z_i],y,\hat{z_i}\bigl)\\
        &~~+\sum_{i<j}\rho(x,[z_i,z_j],\hat{\hat{z}}_{i,j})+\sum_{i<j}\rho(y,[z_i,z_j],\hat{\hat{z}}_{i,j})+\underline{\sum_{i<j}\theta(x,y,[z_i,z_j],\hat{\hat{z}}_{i,j})}.
   \end{align*}
   The underlined terms correspond to $\alpha_{\mu,\nu}(\theta)(x,y,z)$ and the unadorned terms to $\beta_{\w,\epsilon}(\rho)(x,z)+\beta_{\w,\epsilon}(\rho)(y,z)$. Therefore, the pair $\bigl(\alpha_{\mu,\nu}(\theta),\beta_{\w,\epsilon}(\rho)\bigl)$ belongs to the space $\fC^{n+1}_{*_2}(\varphi,\varphi)$ and the maps $\fd^n_{*_2}$ are well-defined. Moreover, we have $$\beta_{\delta\w,\delta\epsilon}\bigl(\beta_{\w,\epsilon}(\rho)\bigl)=\Bigl(\varphi\circ\delta^{n-1}\w+\delta^{n-1}\epsilon\circ\varphi+\delta^{n-1}(\varphi\circ\w+\epsilon\circ\varphi+\delta^{n-2}\rho)\Bigl)=0.$$ Therefore, $\fd^{n+1}_{*_2}\circ\fd^n_{*_2}=0$.
\end{proof}
Let $n\geq0$. Theorem \ref{cohomorph2} allows us to consider the restricted cohomology groups defined by
\begin{equation}
    \mathfrak{H}^n_{*_2}(\varphi,\varphi):=\mathfrak{Z}^n_{*_2}(\varphi,\varphi)/\mathfrak{B}^n_{*_2}(\varphi,\varphi).
\end{equation}

\subsubsection{Deformation of restricted morphisms in characteristic $2$}

Let $\bigl(L,[ \cdot , \cdot ]_L,(\cdot )^{[2]_L}\bigl)$ and $\bigl(M,[ \cdot , \cdot ]_M,(\cdot )^{[2]_M}\bigl)$ be two restricted Lie algebras and  $\varphi:L\rightarrow M$ be a restricted morphism. Let $(\mu_t,\w_t)$ (resp. $(\nu_t,\epsilon_t)$) be a restricted deformation of $L$ (resp. $M$). A restricted deformation of $\varphi$ is a restricted morphism $\varphi_t:\bigl(L[[t]],\mu_t,\w_t\bigl)\rightarrow \bigl(M[[t]],\nu_t,\epsilon_t\bigl)$ given by $$\varphi_t(x):=\displaystyle\sum_{i\geq0}t^i\varphi_i(x),~\varphi_i:L\rightarrow M \text{ linear maps, }\forall x\in L.$$ 
Since $\varphi_t$ is a restricted morphism, it has to satisfy
\begin{align}
    \varphi_t\circ \mu_t(x,y)&=\nu_t(\varphi_t(x),\varphi_t(y))~\forall x,y\in L;\label{eqdefmorph1-2}\\
    \varphi_t\circ\w_t(x)&=\epsilon_t\circ\varphi_t(x)~\forall x\in L.\label{eqdefmorph2-2}
\end{align}
By computations similar to those in Section \ref{242}, we obtain that  
$\bigl((\mu_1,\w_1),(\nu_1,\epsilon_1),\varphi_1\bigl)\in\fZ^2_{*_2}(\varphi,\varphi)$.\\

\noindent\textbf{Obstructions.} Let $\bigl(L,[ \cdot , \cdot ]_L,(\cdot )^{[2]_L}\bigl)$ and $\bigl(M,[ \cdot , \cdot ]_M,(\cdot )^{[2]_M}\bigl)$ be two restricted Lie algebras. Let $(\mu_t^n,\w_t^n)$ (resp. $(\nu_t^n,\epsilon_t^n)$) be a restricted deformation of $L$ (resp. $M$) of order $n>0$, let $\varphi:L\rightarrow M$ be a restricted morphism and let $\varphi_t^n$ be an order $n$  restricted deformation of $\varphi$, that is, $\displaystyle\varphi_t^n=\sum_{i\geq0}^nt^i\varphi_i$. This paragraph is devoted to investigate the obstructions to the extension of the deformation at order $n+1$. Consider
\begin{equation}
\begin{cases}
        (\mu_t^{n+1}=\mu_t^n+t^{n+1}\mu_{n+1},~&~\w_t^{n+1}=\w_t^{n}+t^{n+1}\w_{n+1});\\  (\nu_t^{n+1}=\nu_t^n+t^{n+1}\nu_{n+1},~&~\epsilon_t^{n+1}=\epsilon_t^{n}+t^{n+1}\epsilon_{n+1});\\
        \varphi_t^{n+1}=\varphi^n_t+t^{n+1}\varphi_{n+1},
\end{cases}
\end{equation}
where $(\mu_{n+1},\w_{n+1})\in C^2_{*_2}(L,L)$, $(\nu_{n+1},\epsilon_{n+1})\in C^2_{*_2}(M,M)$ and $\varphi_{n+1}\in C^1_{*_2}(L,M)$. Let $x,y\in L$. Consider the maps
\begin{align}
    \Obs^{(1)}_{n+1}(\varphi)(x,y)&:=\sum_{i=1}^n\varphi_i\circ\mu_{n+1-i}(x,y)+\sum_{\underset{i+j+k=n+1}{i,j,k\leq n}}\nu_k\bigl(\varphi_i(x),\varphi_j(y)\bigl)\\\nonumber
    \Obs^{(2)}_{n+1}(\varphi)(x)&:=\sum_{i=1}^n\varphi_i\circ\w_{n+1-i}(x)+\sum_{\underset{0<j<k}{j+k=n+1}}\bigl[\varphi_j(x),\varphi_k(x)\bigl]\\&~~~+\sum_{i=1}^n\sum_{\underset{j+k=n+1-i}{j<k}}\nu_i\bigl(\varphi_j(x),\varphi_k(x)\bigl).
\end{align}

\begin{lem}
 Let $n\geq 0.$ Then $\Bigl(\Obs^{(1)}_{n+1},\Obs^{(2)}_{n+1}  \Bigl)\in C^2_{*_2}(L,M)$.
\end{lem}

\begin{proof}
    Let $x,y\in L$.
    \begin{align}
        \Obs^{(2)}_{n+1}(\varphi)(x+y)&=\sum_{i=1}^n\varphi_i\bigl(\w_{n+1-i}(x+y)\bigl)\label{obq1}\\
        &+\sum_{\underset{0<j<k}{j+k=n+1}}\Bigl(\underline{\bigl[\varphi_j(x),\varphi_k(x)\bigl]}+\overline{\bigl[\varphi_j(x),\varphi_k(y)\bigl]}+\bigl[\varphi_j(y),\varphi_k(x)\bigl]+\bigl[\varphi_j(y),\varphi_k(y)\bigl]\Bigl)\\
        &+\sum_{i=1}^n\sum_{\underset{j+k=n+1-i}{j<k}}\Bigl(\underline{\nu_i\bigl(\varphi_j(x),\varphi_k(x)\bigl)}+\overline{\nu_i\bigl(\varphi_j(x),\varphi_k(y)\bigl)}\Bigl)\\
        &+\sum_{i=1}^n\sum_{\underset{j+k=n+1-i}{j<k}}\Bigl(\nu_i\bigl(\varphi_j(y),\varphi_k(x)\bigl)+\nu_i\bigl(\varphi_j(y,\varphi_k(y)\bigl)\Bigl).
    \end{align}
    Moreover, expanding the right-hand side term of \eqref{obq1} gives
    \begin{equation}
        \sum_{i=1}^n\varphi_i\bigl(\w_{n+1-i}(x+y)\bigl)=\underline{\sum_{i=1}^n\varphi_i\bigl(\w_{n+1-i}(x)\bigl)}+\overline{\sum_{i=1}^n\varphi_i\bigl(\w_{n+1-i}(y)\bigl)}+\sum_{i=1}^n\varphi_i\bigl(\mu_{n+1-i}(x,y)\bigl).
    \end{equation}
The underlined terms are equal to $\Obs^{(2)}_{n+1}(\varphi)(x)$ while the over-lined terms are equal to $\Obs^{(2)}_{n+1}(\varphi)(y)$. For the remaining terms, we have that
\begin{align*}
    \sum_{\underset{i+j+k=n+1}{i,j,k\leq n}}\nu_k\bigl(\varphi_i(x),\varphi_j(y)\bigl)=&\sum_{\underset{0<j<k}{j+k=n+1}}\Bigl(\bigl[\varphi_j(y),\varphi_k(x)\bigl]+\bigl[\varphi_j(y),\varphi_k(y)\bigl]\Bigl)\\&~~+\sum_{i=1}^n\sum_{\underset{j+k=n+1-i}{j<k}}\Bigl(\nu_i\bigl(\varphi_j(y),\varphi_k(x)\bigl)+\nu_i\bigl(\varphi_j(y,\varphi_k(y)\bigl)\Bigl).
\end{align*}
    Therefore, $\Obs^{(2)}_{n+1}(\varphi)(x+y)=\Obs^{(2)}_{n+1}(\varphi)(x)+\Obs^{(2)}_{n+1}(\varphi)(y)+\Obs^{(1)}_{n+1}(\varphi)(x,y),~\forall x,y\in L$.
\end{proof}
	
 \begin{prop}
     Suppose that $(\mu_t^{n+1}, \nu_t^{n+1},\varphi_t^{n+1})$ is a restricted deformation of the morphism $\varphi$. Then, 
     \begin{equation}\label{hear}
        \alpha_{\mu_{n+1},\nu_{n+1}}(\varphi_{n+1})=\Obs^{(1)}_{n+1}(\varphi),~\text{and }~\beta_{\w_{n+1},\epsilon_{n+1}}(\varphi_{n+1})=\Obs^{(2)}_{n+1}(\varphi).
    \end{equation}
 \end{prop}
 
 \begin{proof}
    We will prove the second part of \eqref{hear}. Suppose that $\varphi_t^{n+1}$ is a restricted morphism. Then
     \begin{equation}\label{obsmorpheq}
         \varphi_t^{n+1}\circ\w_t^{n+1}=\epsilon_t^{n+1}\circ\varphi_t^{n+1}.
     \end{equation}
     Expanding the left-hand side of Eq.\eqref{obsmorpheq} modulo $t^{n+2}$ gives
     \begin{equation}
         \varphi_t^{n+1}\circ\w_t^{n+1}=\sum_{k=0}^{n+1}t^k\sum_{i=0}^{k}\varphi_i\bigl(\w_{k-i}(x)\bigl).
     \end{equation} Therefore, the coefficient of $t^{n+1}$ is
     \begin{equation}\label{midnight}
        \varphi\bigl(\w_{n+1}(x)\bigl)+\varphi_{n+1}\bigl(x^{[2]_L}\bigl)+\sum_{i=1}^{n}\varphi_i\bigl(\w_{n+1-i}(x)\bigl),~~\forall x\in L.
     \end{equation}
    We focus on the right-hand side of Eq.\eqref{obsmorpheq}. First, we mention that similar to Eq.\eqref{scalarformula}, we have for all $x\in M,$
    \begin{equation}
\epsilon_k\Bigl(\sum_{i=0}^{n}\lambda^ix_i\Bigl)=\sum_{i=0}^{n}\lambda^{2i}\epsilon_k(x_i)+\sum_{0\leq i<j\leq n}\nu_k^{i+j}(x_i,x_j),~~\forall (\nu_k,\epsilon_k)\in C_{*_2}^2(M,M).
    \end{equation}
Expanding the right-hand side of Eq.\eqref{obsmorpheq} modulo $t^{n+2}$, we obtain
\begin{equation}
\sum_{i=0}^{n+1}t^i\epsilon_i\Bigl(\sum_{j=0}^{n+1}t^j\varphi_j(x)\Bigl)=\sum_{l=0}^{n+1}t^l\sum_{j=0}^{\lfloor \frac{n+1-l}{2}\rfloor}\epsilon_l\bigl(\varphi_j(x)\bigl)+\sum_{i,l=0}^{n+1}t^{i+l}\sum_{j=0}^{l}\nu_i\bigl(\varphi_j(x),\varphi_{l-j}(x)\bigl),
\end{equation} where $\lfloor\cdot\rfloor$ denotes the floor function. Therefore, the coefficient of $t^{n+1}$ of the right-hand side of Eq.\eqref{obsmorpheq} is given by
\begin{align}\label{gimme}
    \epsilon_{n+1}\circ\varphi(x)+\sum_{\underset{j+k=n+1}{j<k}}\Bigl[\varphi_j(x),\varphi_k(x)\Bigl]+\sum_{i=1}^{n}\sum_{\underset{j+k=n+1-i}{j<k}}\nu_i\bigl(\varphi_j(x),\varphi_{n+1-i-j}(x)\bigl).
\end{align}

Putting \eqref{obsmorpheq}, \eqref{midnight} and \eqref{gimme} together, we obtain
$$\beta_{\w_{n+1},\epsilon_{n+1}}(\varphi_{n+1})=\Obs^{(2)}_{n+1}(\varphi).$$
\end{proof}

\noindent\textbf{Equivalence.} Equivalence of deformations of restricted morphisms in characteristic $p=2$ can be handled as the $p\geq3$ case, see Definition \ref{defiequimorph} and Proposition \ref{propequimorph}.
	\section{Restricted Heisenberg algebras}\label{sectionh}
	
	In this section, we investigate examples based on the Heisenberg Lie algebra of dimension $3$. We study restricted structures of the Heisenberg algebra, then we give an explicit description of the second restricted cohomology spaces with adjoint coefficients.
	Restricted $p$-nilpotent Heisenberg algebras have also been considered in \cite{SU16}, and the general case of restricted Heisenberg Lie algebras of dimension $2n+1$ has been recently investigated in \cite{EFY24}, where the authors provide explicit descriptions of the restricted cohomology with scalar coefficients in order to compute central extensions. For background material on Heisenberg algebras, see \cite{WP17}.
	
	\subsection{Restricted structures  on the Heisenberg Lie algebra, $p\geq3$}
	Let $\mathbb{F}$ be a field of characteristic $p\geq3$. We consider the \textit{Heisenberg algebra} $\h=\text{Span}_{\F} \{x,y,z\}$ defined by the bracket $[x,y]=z.$ This Lie algebra is nilpotent of order $2$, therefore all the $p$-folds brackets on $\h$ vanish. Let $(\cdot)^{[p]}$ be a $p$-map on $\h$. We then have $(u+v)^{[p]}=u^{[p]}+v^{[p]},$ for all $u,v\in \h$. Hence, any $p$-map on $\h$ is $p$-semilinear.

	\begin{prop}\label{classifheisenberg}
		Any $p$-structure on $\h$ is given by $x^{[p]}=\theta(x)z,~y^{[p]}=\theta(y)z,~z^{[p]}=\theta(z)z,$
		with  $\theta: \h\rightarrow \F$ a linear form on $\h$.
	\end{prop}
	
	\begin{proof}
		Using Jacobson's Theorem \ref{jacobson}, it is enough to check that
		$$\left(\ad_{x} \right)^p-\theta(x)\ad_{z}=\left(\ad_{y} \right)^p-\theta(y)\ad_{z}=\left(\ad_{z} \right)^p-\theta(z)\ad_{z}=0$$ to obtain the first claim. The above identities are always true, because $z$ lies in the center of $\h$. Conversely, let $(\cdot)^{[p]}$ be a $p$-map on $\h$. Because of the second condition of  Definition \ref{restdefi}, the image of $(\cdot)^{[p]}$ lies in the center of $\h$, which is one-dimensional and spanned by $z$. 
		Therefore, there exists $\theta:\h\rightarrow \F$ linear such that $x^{[p]}=\theta(x)z,~y^{[p]}=\theta(y)z,~z^{[p]}=\theta(z)z.$
	\end{proof}
	
	\noindent\textbf{Notation.} We will denote a restricted Heisenberg algebra whose $p$-map is given by the linear form $\theta$ by $(\h, \theta)$. 
	
	\begin{rmq}
		Let $v\in(\h,\theta),~v=\alpha x+\beta y + \gamma z,~\alpha,\beta,\gamma\in\F.$ Then, we have $$v^{[p]}=\Bigl(\alpha^p\theta(x)+\beta^p\theta(y)+\gamma^p\theta(z)\Bigl)z.$$
	\end{rmq}

	\begin{lem}\label{heisiso}
		Let $(\h,\theta)$ and $(\h,\theta')$ be two  restricted Heisenberg algebras. Then, any Lie isomorphism $\phi:(\h,\theta)\rightarrow (\h,\theta')$ is of the form 
		\begin{equation}\begin{cases}\label{hiso}
				\phi(x)&=ax+by+cz\\
				\phi(y)&=dx+ey+fz\\
				\phi(z)&=(ae-bd)z,~~ae-bd\neq0,
		\end{cases}\end{equation}
		with $a,b,c,d,e,f\in\F$. Moreover, $\phi$ is a restricted Lie isomorphism if and only if
		\begin{equation}\begin{cases}\label{resthiso}
				\theta(x)u&=a^p\theta'(x)+b^p\theta'(y)+c^p\theta'(z)\\
				\theta(y)u&=d^p\theta'(x)+e^p\theta'(y)+f^p\theta'(z)\\
				\theta(z)u&=u^p\theta'(z),
		\end{cases}\end{equation}
		where $u:=ae-bd\neq0$.
	\end{lem}
	
	\begin{proof}
		A Lie isomorphism $\phi:\h\rightarrow\h$ must satisfy $\phi([v,w])=[\phi(v),\phi(w)],\text{ for all }v,w\in\h,$ as well as $\det(\phi)=0$. Applying these conditions to an arbitrary linear map $\phi:\h\rightarrow\h$, we obtain Conditions (\ref{hiso}). Then, $\phi$ is a restricted map on $\h$ if and only if $\phi\left(v^{[p]}\right)=\phi(v)^{[p]'}$, for all $v\in \h$ and with $(\cdot)^{[p]'}$ the $p$-map on $\h$ given by the linear form $\theta'$. We obtain Conditions (\ref{resthiso}) by evaluating this equation on the basis elements of $\h$. For example, $\phi\bigl(x^{[p]}\bigl)=\phi(x)^{[p]'}$ is equivalent to $\theta(x)u=\theta'(x)+b^p\theta'(y)z+c^p\theta'(z).$ The two other equations are obtained in a similar way.
	\end{proof}
	
	\begin{thm}\label{heisclass}
		There are three non-isomorphic restricted Heisenberg algebras, respectively given by the linear forms $\theta=0,$ $\theta=x^*$ and $\theta=z^*$.
	\end{thm}
	
	\begin{proof}
		\begin{itemize}
			\item[$\bullet$] First, we will show that $(\h,x^*)$ is isomorphic to $(\h,y^*)$. By setting $\theta=x^*$ and $\theta'=y^*,$ Conditions (\ref{resthiso}) reduce to $\{u=b^p,~e^p=0\}$. We choose $e=0, b\neq 0$ and $d=-b^{p-1}$ to build a restricted isomorphism between $(\h,x^*)$ and $(\h,y^*)$.
			\item[$\bullet$] Let $\theta=0$ and $\theta'=x^*$. Then, Conditions (\ref{resthiso}) reduce to $\{a^p=0,~d^p=0\}$. But, this is impossible since $u=ae-bd\neq0$. Therefore, $(\h,0)$ and $(\h,x^*)$ are not isomorphic.
			\item[$\bullet$] Let $\theta=0$ and $\theta'=z^*$. Then, Conditions (\ref{resthiso}) reduce to $\{c^p=0,~f^p=0,~u^p=0\}$. But, this is impossible since $u\neq0$. Therefore, $(\h,0)$ and $(\h,z^*)$ are not isomorphic.
			\item[$\bullet$] Let $\theta=x^*$ and $\theta'=z^*$. Then, Conditions (\ref{resthiso}) reduce to $\{c^p=u,~f^p=0,~u^p=0\}$. But, this is impossible since $u\neq0$. Therefore, $(\h,x^*)$ and $(\h,z^*)$ are not isomorphic.
		\end{itemize}
	\end{proof}
	\begin{rmq} The restricted algebras $(\h,0)$ and $(\h,x^*)$ appeared in \cite{SU16} and are $p$-nilpotent.
	\end{rmq}
	
	
	In the sequel, we compute the second restricted cohomology groups of the restricted Heisenberg Lie algebras with adjoint coefficients. Let $\theta$ be a linear form on the (ordinary) Heisenberg Lie algebra. We denote by $(\h,\theta)$ the restricted Heisenberg Lie algebra obtained with $\theta$ (see Proposition \ref{classifheisenberg}). We also denote by $H^2_*(\h,\theta):=H^2_*((\h,\theta),(\h,\theta))$ the second restricted cohomology group of $(\h,\theta)$ with adjoint coefficients.
	
	\subsubsection{Restricted cohomology, case  $p>3$}
	
	Let $\F$ be a field of characteristic $p>3$ and let $\varphi\in C_{\text{CE}}^2(\h,\h)$.
	Since the (ordinary) Heisenberg Lie algebra $\h$ is nilpotent of order $2$ and $p>3$, any $p$-semilinear map $\w:\h\rightarrow\h$ satisfies the $(*)$-property with respect to $\varphi$. 
	
	\begin{lem}\label{ordcocycle}
		Let $\h$ be the ordinary Heisenberg Lie algebra. Let $\varphi\in C_{\text{CE}}^2(\h,\h)$ given by
		\begin{equation}\label{stdrphi}
			\begin{cases}
				\varphi(x,y)&=ax+by+cz\\
				\varphi(x,z)&=dx+ey+fz\\
				\varphi(y,z)&=gx+hy+iz
			\end{cases}
		\end{equation}
		with parameters $a,b,c,d,e,f,g,h,$ belonging to $\F$. Then, $\varphi$ is a $2$-cocycle of the Chevalley-Eilenberg cohomology if and only if $h=-d$.
	\end{lem}
	
	\begin{proof}
		The only non-trivial 2-cocycle condition on the basis $\{x,y,z\}$ of $\h$ is
		\begin{equation}
			\varphi([x,y],z)-\varphi([x,z],y)+\varphi([y,z],x)=[x,\varphi(y,z)]-[y,\varphi(x,z)]+[z,\varphi(x,y)],
		\end{equation}
		which reduces to $(h+d)z=0.$
	\end{proof}
	Let $(\varphi,\w)\in C^2_*(\h,\h)$. As the $k$-folds brackets vanish for $k>2$ and $p>3$, we have
	\begin{equation}\label{pcocyl}
		\ind^2(\varphi,\w)(v,w)=\varphi\bigl(v,w^{[p]}\bigl)+[v,\w(w)],~\forall v,w\in \h.
	\end{equation}
	
	\begin{lem}\label{hrescocy}
		The restricted $2$-cocycles for $(\h,\theta)$ are given by pairs $(\varphi,\w)$, where
		\begin{itemize}
			\item[$\bullet$] \underline{Case $\theta=0$}:  
			\begin{equation}\begin{cases}
					\varphi(x,y)&=ax+by+cz\\
					\varphi(x,z)&=dx+ey+fz\\
					\varphi(y,z)&=gx-dy+iz\\
				\end{cases}
				~~~~~~~~~~
				\begin{cases}
					\w(x)&=\gamma z\\
					\w(y)&=\epsilon z\\
					\w(z)&=\kappa z;\\
			\end{cases}\end{equation}
			
			\item[$\bullet$] \underline{Case $\theta=x^*$}:  
			\begin{equation}\begin{cases}
					\varphi(x,y)&=ax+by+cz\\
					\varphi(x,z)&=fz\\
					\varphi(y,z)&=iz\\
				\end{cases}
				~~~~~~~~~~
				\begin{cases}
					\w(x)&=ix-fy+\gamma z\\
					\w(y)&=\epsilon z\\
					\w(z)&=\kappa z;\\
			\end{cases}\end{equation}
			
			\item[$\bullet$] \underline{Case $\theta=z^*$}: 
			\begin{equation}\begin{cases}
					\varphi(x,y)&=ax+by+cz\\
					\varphi(x,z)&=fz\\
					\varphi(y,z)&=iz;\\
				\end{cases}
				~~~~~~~~~~
				\begin{cases}
					\w(x)&=\gamma z\\
					\w(y)&=\epsilon z\\
					\w(z)&=ix-fy+\kappa z,\\
			\end{cases}\end{equation} where all the parameters $a,b,c,d,e,f,h,i,\gamma,\epsilon,\kappa$ belong to $\F$.
		\end{itemize}
	\end{lem}
	
	\begin{proof}
		Let $\varphi\in Z_{\text{CE}}^2(\h,\h)$ given by Lemma \ref{ordcocycle}, let $\theta$ be a linear form on $\h$ and let $\w:\h\rightarrow\F$ be a map having the $(*)$-property w.r.t $\varphi$, given on the basis of $\h$ by
		\begin{equation}
			\begin{cases}
				\w(x)&=\alpha x+\beta y+\gamma z\\
				\w(y)&=\lambda x+\mu y+\epsilon z\\
				\w(z)&=\delta x+\eta y+\kappa z,\\
			\end{cases}
		\end{equation}
		with $\alpha,\beta,\gamma,\lambda,\mu,\epsilon,\delta,\eta,\kappa$ coefficients in $\F$. Moreover, suppose that $(\varphi,\w)\in Z_*^2(\h,\theta)$.
		
		\begin{itemize}
			\item Let $\theta=0$. By evaluating Eq.(\ref{pcocyl}) on basis elements $\{x,y,z\}$ of $(\h,0)$, we obtain $\beta=\lambda=\mu=\nu=\delta=\alpha=0$.

			\item Let $\theta=x^*$. By evaluating Eq.(\ref{pcocyl}) on elements of the basis $\{x,y,z\}$ of $(\h,x^*)$, we obtain $\beta=-f,~\alpha=i$ and $d=e=g=\lambda=\mu=\nu=\delta=0$.
			\item The case $\theta=z^*$ is analog to the case $\theta=x^*$.
		\end{itemize}
	\end{proof}

	\begin{lem}\label{hrestcob}
		The restricted $2$-coboundaries for $(\h,\theta)$ are given by pairs $(\varphi,\w)$, where
		
		$$\begin{cases}
			\varphi(x,y)&=Ax+By+\tilde{C}z\\
			\varphi(x,z)&=-Hz\\
			\varphi(y,z)&=Gz,\\
		\end{cases}$$ with $A,B,\tilde{C},G,H$ belonging to $\F$ and
		\begin{itemize}
			\item[$\bullet$] \underline{Case $\theta=0$}: $\w=0;$ 
			
			\item[$\bullet$] \underline{Case $\theta=x^*$}: $\w(x)=Gx+Hy+Iz,~\w(y)=\w(z)=0;$
			
			\item[$\bullet$] \underline{Case $\theta=z^*$}:
			$\w(x)=\w(y)=0,~\w(z)=Gx+Hy+Iz.$
		\end{itemize}

	\end{lem}
	
	\begin{proof}
		Let $\varphi\in C^2_{\text{CE}}(\h,\h),$ given on the basis of $\h$ by Eq.(\ref{stdrphi}). Suppose that $\varphi=d^1_{\text{CE}}\psi,$ with $\psi:\h\rightarrow\h$ given by
		\begin{equation}
			\begin{cases}
				\psi(x,y)&=Ax+By+Cz\\
				\psi(x,z)&=Dx+Ey+Fz\\
				\psi(y,z)&=Gx+Hy+Iz,
			\end{cases}
		\end{equation}
		with $A,B,C,D,E,F,G,H,I\in \F$. Using the coboundary condition $\varphi=d^1_{\text{CE}}\psi,$ we show that $d=e=g=h=0,~a=i=G,~b=-f=H,~c=\Tilde{C},$ with $\tilde{C}=I-E-A$. For the restricted part, suppose that $(\varphi,\w)\in B^2_*(\h,\theta)$. The coboundary condition is then given by
		\begin{equation}\label{cob}
			\w(u)=\psi\bigl(u^{[p]}\bigl)-\ad_u^{p-1}\circ~\psi(u),~\forall u\in \h.
		\end{equation}
		By evaluating Eq.(\ref{cob}) on the basis of $\h$, we obtain
		
		\begin{equation}
			\begin{cases}
				\w(x)&=\theta(x)(Gx+Hy+Iz)\\
				\w(y)&=\theta(y)(Gx+Hy+Iz)\\
				\w(z)&=\theta(z)(Gx+Hy+Iz).
			\end{cases}
		\end{equation}
		Choosing $\theta$ in $\{0,x^*,z^*\}$, we get the result.
	\end{proof}

	\begin{thm}\label{cohogroupp}
		Let $\F$ be a field of characteristic $p>3$. We have $\dim_{\F}\left(H^2_*(\h,0)\right)=8$ and $\dim_{\F}\left(H^2_*(\h,x^*)\right)=\dim_{\F}\left(H^2_*(\h,z^*)\right)=4.$
		\begin{itemize}
			\item[$\bullet$] A basis for $H^2_*(\h,0)$ is given by $\{(\varphi_1,
			0), (\varphi_2,0), (\varphi_3,0), (\varphi_4,0), (\varphi_5,0), (0,\w_1), (0,\w_2), (0,\w_3)\}$, with \begin{align*}
				\varphi_1(x,z)&=z;~\varphi_2(y,z)=z;~\varphi_3(x,z)=-\varphi_3(y,z)=x;~\varphi_4(x,z)=y;~\varphi_5(y,z)=y;\\
				\w_1(x)&=z;~\w_2(y)=z;~\w_3(z)=z.
			\end{align*}
			(We only write non-zero images).
			\item[$\bullet$] A basis for $H^2_*(\h,x^*)$ is given by $\{(\varphi_1,
			0), (\varphi_2,0), (0,\w_1), (0,\w_2)\}$, with
			$$\varphi_1(x,y)=x;~\varphi_2(x,y)=y;~
			\w_1(y)=z;~\w_2(z)=z.$$
			
			\item[$\bullet$] A basis for $H^2_*(\h,z^*)$ is given by $\{(\varphi_1,
			0), (\varphi_2,0), (0,\w_1), (0,\w_2)\}$, with
			$$\varphi_1(x,y)=x;~\varphi_2(x,y)=y;~
			\w_1(y)=z;~\w_2(x)=z.$$
		\end{itemize}
	\end{thm}
	\begin{proof}
		
		With Lemma \ref{hrestcob}, we deduce that $\{(\varphi_6,0),(\varphi_7,0),(\varphi_8,0)\}$ is a basis of $B^2_*(\h,0)$, with
		$$\varphi_6(x,y)=x,~\varphi_6(y,z)=z;~~\varphi_7(x,y)=y,~\varphi_7(x,z)=-z;~~\varphi
		_8(x,y)=z.$$
		Using Lemma \ref{hrescocy}, we complete the above basis in a basis for $Z_*^2(\h,0)$ and therefore we find the basis of $H^2_*(\h,0)$. The two other cases with $\theta=x^*$ or $\theta=z^*$ are similar.
	\end{proof}

    \noindent\textbf{An example of computation of morphism cocycles.} Let $L=(\h,x^*)$ and $M=(\h,z^*)$. Consider the restricted morphism $\varphi:L\rightarrow M$ given by
    $$\varphi(x)=z;~~\varphi(y)=x+y;~~\varphi(z)=0.$$
The general form of the elements of $Z_*^2(L,L)$ and $Z_*^2(M,M)$ are given in Lemma \ref{hrescocy}. Consider $(\mu,\w)\in Z_*^2(L,L)$ and $(\nu,\epsilon)\in Z_*^2(M,M)$ given on the basis $\{x,y,z\}$ by
$$\mu(x,z)=z,~\w(y)=z;~~  \nu(x,y)=x,~\epsilon(y)=z.  $$

\begin{prop}
    With the above data, the space $\Ker\bigl(\alpha_{\mu,\nu}\bigl)\cap\Ker\bigl(\beta_{\w,\epsilon}\bigl)$ is spanned by $\{\theta_1,\theta_2,\theta_3\}$, where
    $$ \theta_1(y)=x,~\theta_1(z)=z;~~ \theta_2(y)=y;~~\theta_3(y)=z.$$ (We only write non-zero images.)
\end{prop}

	\subsubsection{Restricted cohomology, case $p=3$}
	
	Let $\F$ be a field of characteristic $3$ and let $\varphi\in C_{\text{CE}}^2(\h,\h)$. Then, a map $\w:\h\rightarrow\h$ has the $(*)$-property with respect to $\varphi$ if and only if
	\begin{equation}\label{3st}
		\w(u+v)=\w(u)+\w(v)+2\left(\varphi([u,v],u)+[\varphi(u,v),u]\right)+\varphi([u,v],v)+[\varphi(u,v),v],~\forall u,v\in \h.
	\end{equation}
	
	Let $\theta$ be a linear form on $\h$ and let $(\varphi,\w)\in C_{*}^2(\h,\theta)$. We recall that $\h$ is endowed with a $3$-map $(\cdot)^{[3]}$ given by $\theta$ (see Proposition \ref{classifheisenberg}). Since $p=3$, we have
	\begin{equation}\label{3ind}
		\ind^2(\varphi,\w)(u,v)=\varphi\bigl(u,v^{[3]}\bigl)-\bigl[\varphi([u,v],v),v\bigl]+[u,\w(v)].
	\end{equation}
	
	\begin{lem}\label{3hrescocy}
		The restricted $2$-cocycles for $(\h,\theta)$ are given by pairs $(\varphi,\w)$, where
		\begin{itemize}
			\item[$\bullet$] \underline{Case $\theta=0$}:
			\begin{equation}\begin{cases}
					\varphi(x,y)&=ax+by+cz\\
					\varphi(x,z)&=dx+ey+fz\\
					\varphi(y,z)&=gx-dy+iz\\
				\end{cases}
				~~~~~~~~~~
				\begin{cases}
					\w(x)&=-ex+\gamma z\\
					\w(y)&=dy+\epsilon z\\
					\w(z)&=\kappa z\\
			\end{cases}\end{equation}
			
			\item[$\bullet$] \underline{Case $\theta=x^*$}: same as Lemma \ref{hrescocy};
			\item[$\bullet$] \underline{Case $\theta=z^*$}: same as Lemma \ref{hrescocy}; 
			
			where all the parameters $a,b,c,d,e,f,h,i,\gamma,\epsilon,\kappa$ belong to $\F$.
		\end{itemize}
	\end{lem}
	
	\begin{proof}
		Similar to  Lemma \ref{hrescocy}, but using Eq.\eqref{3ind}.
	\end{proof}
	A similar computation shows that the restricted $2$-coboundaries are the same as in Lemma \ref{hrestcob}. The $(*)$-property is given by Eq.(\ref{3st}) if $p=3$.
	
	\begin{thm}\label{cohogroups3} Let $\F$ be a field of characteristic $p=3$.
		We have $\dim_{\F}\left(H^2_*(\h,0)\right)=8$ and $\dim_{\F}\left(H^2_*(\h,x^*)\right)=\dim_{\F}\left(H^2_*(\h,z^*)\right)=4.$
		\begin{itemize}
			\item[$\bullet$] A basis for $H^2_*(\h,0)$ is given by $\{(\varphi_1,
			\w_1), (\varphi_2,\w_2), (\varphi_3,0), (\varphi_4,0), (\varphi_5,0), (0,\w_3), (0,\w_4), (0,\w_5)\}$, with
			\begin{align*}
				\varphi_1(x,z)&=-\varphi_1(y,z)=x,~\w_1(y)=x;~\varphi_2(x,z)=y,~\w_2(x)=x;~\varphi_3(y,z)=z;\\
				\varphi_4(x,z)&=z;~\varphi_5(y,z)=y;~\w_3(x)=\w_4(y)=\w_5(z)=z.
			\end{align*}
			(We only write non-zero identities).
			\item[$\bullet$] A basis for $H^2_*(\h,x^*)$ is given by $\{(\varphi_1,
			0), (\varphi_2,0), (0,\w_1), (0,\w_2)\}$, with
			$$\varphi_1(x,y)=x;~\varphi_2(x,y)=y;~
			\w_1(y)=z;~\w_2(z)=z.$$
			
			\item[$\bullet$] A basis for $H^2_*(\h,z^*)$ is given by $\{(\varphi_1,
			0), (\varphi_2,0), (0,\w_1), (0,\w_2)\}$, with
			$$\varphi_1(x,y)=x;~\varphi_2(x,y)=y;~
			\w_1(y)=z;~\w_2(x)=z.$$
		\end{itemize}
	\end{thm}
	
	\subsection{Restricted structures on the Heisenberg Lie algebra, $p=2$}
	Let $\F$ be an algebraically closed field of characteristic $2$. Recall that the Heisenberg algebra $\h$ is spanned by elements $x,y,z$ with bracket given by $[x,y]=z.$ In characteristic $p=2$, $\h$ is isomorphic to $\Sl_2$.    Let $(\cdot)^{[2]}$ be a $2$-mapping on $\h$. Then, we have
	\begin{align}
		\nonumber(x+y)^{[2]}&=x^{[2]}+y^{[2]}+z;\\
		(x+z)^{[2]}&=x^{[2]}+z^{[2]};\\\nonumber
		(y+z)^{[2]}&=y^{[2]}+z^{[2]}.\nonumber
	\end{align}
	
	Therefore, the $2$-mapping is not $2$-semilinear. Let $u=ax+by+cz\in\h$, $a,b,c\in \F$. Then
	\begin{equation}
		u^{[2]}=(ax+by+cz)^{[2]}=a^2x^{[2]}+b^2y^{[2]}+c^2z^{[2]}+abz.
	\end{equation}
	
	Because of the second condition of the Definition \ref{restlie2}, the image of $(\cdot)^{[2]}$ lies in the center of $\h$, which is one-dimensional and spanned by $z$. Therefore, it exists $\theta:\h\rightarrow \F$ linear such that $x^{[2]}=\theta(x)z,~y^{[2]}=\theta(y)z,~z^{[2]}=\theta(z)z.$ We deduce the following result.
	
	\begin{thm}\label{heisclass2}
		There are two non-isomorphic restricted Heisenberg algebras in characteristic $2$, respectively given by the linear forms $\theta=0$ and $\theta=z^*$.
	\end{thm}
	
	We compute the second restricted cohomology groups of the restricted Heisenberg algebras with adjoint coefficients. 
	Let $u,v\in\h$ and $(\varphi,\w)\in C_{*_2}^2(\h,\theta)$. The restricted part of the $2$-cocycle condition is given by
	\begin{equation}\label{eq1}
		\varphi\left(u,\theta(v)z\right)+[u,\w(v)]+[\varphi(u,v),v]+\varphi([u,v],v)=0.
	\end{equation}
	The restricted part of the $2$-coboundary condition is given for $\psi\in C^1_{*_2}(L,L)$ by
	\begin{equation}\label{eq2}
		\w(u)=[\psi(u),u]+\psi\bigl(\theta(u)z\bigl).
	\end{equation}
	
	By applying Eqs.(\ref{eq1}) and (\ref{eq2}) to an arbitrary pair $(\varphi,\w)\in C^2_{\text{CE}}(\h,\theta)$,  we obtain the general form of the $2$-cocycles and $2$-coboundaries.
	
	\begin{lem}\label{2hrescocy}
		The restricted $2$-cocycles for $(\h,\theta)$ are given by pairs $(\varphi,\w)$, where
		\begin{itemize}
			\item[$\bullet$] \underline{Case $\theta=0$}:
			
			\begin{equation}\begin{cases}
					\varphi(x,y)&=ax+by+cz\\
					\varphi(x,z)&=fz\\
					\varphi(y,z)&=iz\\
				\end{cases}
				~~~~~~~~~~
				\begin{cases}
					\w(x)&=(b+f)x+\gamma z\\
					\w(y)&=(a+i)y+\epsilon z\\
					\w(z)&=\kappa z\\
			\end{cases}\end{equation}
			
			\item[$\bullet$] \underline{Case $\theta=z^*$}:
			
			\begin{equation}\begin{cases}
					\varphi(x,y)&=ax+by+cz\\
					\varphi(x,z)&=fz\\
					\varphi(y,z)&=iz\\
				\end{cases}
				~~~~~~~~~~
				\begin{cases}
					\w(x)&=(b+f)x+\gamma z\\
					\w(y)&=(a+i)y+\epsilon z\\
					\w(z)&=ix+fy+\kappa z,\\
			\end{cases}\end{equation} where all the parameters $a,b,c,d,e,f,h,i,\gamma,\epsilon,\kappa$ belong to $\F$.
		\end{itemize}
	\end{lem}
	
	\begin{lem}\label{2hrestcob}
		The restricted $2$-coboundaries for $(\h,\theta)$ are given by pairs $(\varphi,\w)$, where
		
		$$\begin{cases}
			\varphi(x,y)&=Ax+By+\tilde{C}z\\
			\varphi(x,z)&=Hz\\
			\varphi(y,z)&=Gz,\\
		\end{cases}$$ where $A,B,\tilde{C},D,E,G,H$ belong to $\F$ and
		\begin{itemize}
			\item[$\bullet$] \underline{Case $\theta=0$}: $\w(x)=Ez,~\w(y)=Dz,~\w(z)=0;$ 
			
			\item[$\bullet$] \underline{Case $\theta=z^*$}:
			$\w(x)=Ez,~\w(y)=Dy,~\w(z)=Gx+Hy+Iz.$
		\end{itemize}
		
	\end{lem}
	
	Using Lemmas \ref{2hrescocy} and \ref{2hrestcob}, we are able to compute a basis for the second cohomology spaces.
	
	\begin{thm}\label{cohogroup2}
		We have $\dim_{\F}\left(H^2_{*_2}(\h,0)\right)=3$ and $\dim_{\F}\left(H^2_{*_2}(\h,z^*)\right)=2.$
		
		\begin{itemize}
			\item[$\bullet$] A basis for $H^2_{*_2}(\h,0)$ is given by $\{(\varphi_1,
			\w_1), (\varphi_2,\w_2), (0,\w_3)\}$, with
			$$\varphi_1(y,z)=z;~\varphi_2(x,z)=z;~\w_1(y)=y;~\w_2(x)=x;~\w_3(z)=z.$$
			(We only write non-zero images).
			\item[$\bullet$] A basis for $H^2_{*_2}(\h,z^*)$ is given by $\{(\varphi_1,
			\w_1), (\varphi_2,\w_2)\}$, with
			$$\varphi_1(x,y)=x;~\varphi_2(x,y)=y;~
			\w_1(y)=y;~\w_2(x)=x.$$
		\end{itemize}
	\end{thm}

\noindent\textbf{Example of deformations $(p=2)$.} Consider the restricted Lie algebra $(\h,0)$. The non trivial 2-cocycle are  $\{(\varphi_1, \w_1), (\varphi_2,\w_2), (0,\w_3)\}$, see Thm. \ref{cohogroup2}. First, using the 2-cocycle $(0,\w_3)$, the algebra $(\h,0)$ deforms into $(\h,z^*)$. Then, using the 2-cocycle $(\varphi_2,\w_2)$, a deformation of order 1 is given by the bracket 
\begin{equation}
    [x,y]_t=z,~~[x,z]_t=tz,~~[y,z]_t=0;
\end{equation} and the $2$-map
\begin{equation}
    x^{[2]_t}=tx,~~y^{[2]_t}=z^{[2]_t}=0.
\end{equation}
One can readily check that the deformed algebra is indeed a restricted Lie algebra, for example, we have
 $$[[y,x]_t,x]_t=[z,x]_t=tz=[y,tx]_t=[y,x^{[2]_t}]_t.   $$


\noindent\textbf{Acknowledgement.} QE would like to thank S. Bouarroudj for many stimulating discussions and for his support.

\end{document}